\documentclass[reqno]{amsart} \usepackage{enumerate} \usepackage{amssymb}
\usepackage{amsmath} \usepackage{float} \usepackage{url} \usepackage{graphicx}
\usepackage{fancyhdr}

\def\maxexp{10}
\def\suppress#1{}
\def\nl{\smallskip\noindent}
\newcommand{\ad}{\operatorname{ad}}
\newcommand{\Aut}{\mathrm{Aut}}
\newcommand{\BB}{{\mathrm B}_2\SC}
\newcommand{\ch}{^\vee}
\newcommand{\chr}{\operatorname{char}}
\newcommand{\Cent}{\mathrm{C}}
\newcommand{\Der}{\operatorname{Der}}
\newcommand{\End}{\operatorname{End}}

\newcommand{\F}{{\mathbb F}}
\newcommand{\GAP}{{\sc GAP}}
\newcommand{\GF}{\operatorname{GF}}

\newcommand{\Lie}{\operatorname{Lie}}
\newcommand{\LZ}{L_{\mathbb Z}}
\newcommand{\LF}{L_{\mathbb F}}
\newcommand{\Magma}{{\sc Magma}}
\newcommand{\N}{{\mathbb N}}
\newcommand{\Ad}{\mathrm{^{ad}}}
\newcommand{\SC}{\mathrm{^{sc}}}
\newcommand{\rint}{\mathrm{^{(a)}}}
\newcommand{\rnk}{\operatorname{rk}}

\newcommand{\tO}{O^{\sim}}
\newcommand{\Z}{{\mathbb Z}}

\newcommand{\An}{\mathrm{A}_n}
\newcommand{\Bn}{\mathrm{B}_n}
\newcommand{\Cn}{\mathrm{C}_n}
\newcommand{\Dn}{\mathrm{D}_n}
\newcommand{\En}{\mathrm{E}_{6,7,8}}
\newcommand{\Fn}{\mathrm{F}_4}
\newcommand{\Gn}{\mathrm{G}_2}

\newcommand{\tmpalgocap}{}
\newcommand{\tmpalgolbl}{}
\newenvironment{algorithm}[3]{%
  \renewcommand{\tmpalgocap}{#2}%
  \renewcommand{\tmpalgolbl}{#3}%
  \begin{figure}%
  \begin{tabbing}%
  \quad\quad\=\quad\=\quad\=\quad\=\kill%
  {\sc #1} \\%
}{%
  \end{tabbing}%
  \caption{\tmpalgocap}\label{alg_\tmpalgolbl}%
  \end{figure} %
}
\newcounter{algln}
\newcommand{\lnreset}{\setcounter{algln}{0}}
\newcommand{\lnp}{\addtocounter{algln}{1} {\footnotesize \arabic{algln}}}

\author[Arjeh~M.~Cohen]{Arjeh~M.~Cohen}
\address[Arjeh~M.~Cohen]{
Department of Mathematics and Computer Science\\
Technische Universiteit Eindhoven\\
P.O. Box 513, 5600 MB Eindhoven, Netherlands}
\email{a.m.cohen@tue.nl}

\author[Dan~Roozemond]{Dan~Roozemond}
\address[Dan~Roozemond]{
Department of Mathematics and Computer Science\\
Technische Universiteit Eindhoven\\
P.O. Box 513, 5600 MB Eindhoven, Netherlands}
\email{d.a.roozemond@tue.nl}

\newtheorem{theorem}{Theorem}
\newtheorem{lemma}[theorem]{Lemma}
\newtheorem{proposition}[theorem]{Proposition}

\title{Computing Chevalley bases in small characteristics}

\begin{document}

\begin{abstract}
Let $L$ be the Lie algebra of a simple algebraic group defined over a field
$\F$ and let $H$ be a split maximal toral subalgebra of $L$.  Then $L$ has a
Chevalley basis with respect to $H$.  If $\chr(\F) \neq 2,3$, it is known how
to find it.  In this paper, we treat the remaining two characteristics.  To
this end, we present a few new methods, implemented in {\Magma}, which vary
from the computation of centralisers of one root space in another to the
computation of a specific part of the Lie algebra of derivations of $L$.
\end{abstract}

\maketitle

\begin{center}
\emph{On the occasion of the distinguished birthdays \\
of our distinguished colleagues John Cannon and Derek Holt}
\end{center}

\bigskip

\section{Introduction}

\subsection{The main result}\label{sec_intro_mainresult}
For computational problems regarding a split reductive algebraic group $G$
defined over a field $\F$, it is often useful to calculate within its Lie
algebra $L$ over $\F$. For instance, the conjugacy question for two split
maximal tori in $G$ can often be translated to a conjugacy question for two
split Cartan subalgebras of $L$.  Here, a \emph{Cartan subalgebra} $H$ of $L$
is understood to be a maximal toral subalgebra, that is, it is commutative,
left multiplication by each of its elements is semisimple (i.e., has a
diagonal form with respect to a suitable basis over a large enough extension
field of $\F$), and it is maximal (with respect to inclusion) among
subalgebras of $L$ with these properties; it is called \emph{split} (or
$\F$-\emph{split}) if left multiplication by $h$, denoted $\ad_h$, has a
diagonal form with respect to a suitable basis over $\F$ for every $h\in H$.
Such a Cartan subalgebra is the Lie algebra of a split maximal torus in $G$.

The conjugacy question mentioned above can be answered by finding Chevalley
bases with respect to each split Cartan subalgebra, so the transformation from
one basis to the other is an automorphism of $L$, and subsequently adjusting
the automorphism with the normalizer of one Cartan subalgebra so as to obtain
an element of $G$.  In this light, it is of importance to have an algorithm
finding a Chevalley basis (see Section \ref{sec:CLA} for a precise
definition). Such algorithms have been discussed for the case where the
characteristic of the underlying field is distinct from $2$ and from
$3$. However, the latter two characteristics are the most important ones for
finite simple groups arising from algebraic groups, so there is a need for
dealing with these special cases as well. This is taken care of by the
following theorem, which is the main result of this paper.  

\begin{theorem}\label{th:main}
Let $L$ be the Lie algebra of a split simple algebraic group of rank $n$
defined over an effective field $\F$.  Suppose that $H$ is an $\F$-split
Cartan subalgebra of $L$.  If $L$ is given by means of a multiplication table
with respect to an $\F$-basis of $L$ and $H$ is given by means of a spanning
set, then there is a Las Vegas algorithm that finds a Chevalley basis of $L$
with respect to $H$. If $\F=\F_q$, this algorithm needs $\tO(n^{\maxexp}\log(q)^4)$
elementary operations.
\end{theorem}

Here $\tO(N)$ means $O(N(\log(N))^c)$ for some constant $c$.  Recall (e.g.,
from \cite[Introduction]{Shparlinski99Book}) that arithmetic operations in
$\F$ are understood to be addition, subtraction, multiplication, division, and
equality testing.  If $\F$ is the field $\F_q$ of size $q$, these all take
$\tO(\log(q))$ elementary operations.  Performing standard linear algebra
arithmetic, that is, operations on matrices of size $m$, like multiplication,
determinant, and kernel (solving linear equations), takes $O(m^3)$ arithmetic
operations.

Better estimates than those of the theorem are conceivable, for instance
because better bounds on matrix multiplication exist. However, our primary
goal was to establish that the algorithm is polynomial in
$n\log(q)$. Moreover, in comparison to the dimension $O(n^2)$ of $L$ or the
estimate $O(n^6)$ for arithmetic operations needed for multiplying two
elements of $L$, the high exponent of $n$ in the timing looks more reasonable
than may seem at first sight.


The proof of Theorem \ref{th:main} rests on Algorithm \ref{alg_chevbas}, which
is really an outline of an algorithm further specified in the course of the
paper. The algorithm is implemented in {\Magma} \cite{Magma}.
We intend to make the implementation public as a {\Magma} package once the
code has been cleaned up.

The algorithm is mostly deterministic.  However, in some instances where $\F$
is of characteristic $2$ (such as Method [$\BB$] and the case where $L$ is of
type $\mathrm D_4$; see Sections \ref{sec_multdimsol_D4},
\ref{sec_multdimsol_B2SC}, and \ref{sec_multdimsol_CnSC}) we use the Meat-axe
(cf.~\cite{Holt,HoltEickOBrien}) for finding a particular submodule of a given
module.  For finite fields, the Meat-axe algorithm is analysed \ in
\cite{Ronyai} and \cite[Section 2]{Holt}: irreducible submodules of a finite
$L$-module of dimension $m$ over $\F_q$ can be found in Las Vegas time
$\tO(m^3 \log(q))$ (in Section \ref{sec_intro_computing_ideals} below it is explained why
this result can be applied to Lie algebras).  For infinite fields, Meat-axe
procedures are known; however, we know of no proof
of polynomiality in the literature.

Algorithm \ref{alg_chevbas} assumes that besides $L$ and
$H$ the root datum $R$ of the underlying group is known (see Section \ref{sec_intr_rootdata}).
However, in
Section \ref{sec_conclusion} we show that this root datum can be
determined by running the algorithm a small number of times.

Thanks to the characterization of Lie algebras of split
reductive algebraic groups described in Theorem \ref{thm_liealg_has_chevbas}
below, we can view the Lie algebras in Theorem \ref{th:main} as Chevalley Lie algebras,
which are defined below.

Our treatment of Lie algebras and the corresponding algebraic groups rests on
the theory developed mainly by Chevalley and available in the excellent books
by Borel \cite{Bor69}, Humphreys \cite{Hum75}, and Springer \cite{Springer98}.
Our computational set-up is as in \cite{CMT04}.

\subsection{Root Data}\label{sec_intr_rootdata}
Split reductive algebraic groups are determined by their fields of definition
and their root data. The latter is of importance to the corresponding Lie
algebra and will therefore be discussed first. Throughout this paper we let $R
= (X, \Phi, Y, \Phi\ch)$ be a \emph{root datum of rank} $n$ as defined in
\cite{CMT04}.  This means $X$ and $Y$ are dual free $\Z$-modules of dimension
$n$ with a bilinear pairing $\langle \cdot, \cdot \rangle : X \times Y
\rightarrow \Z$; furthermore, $\Phi$ is a finite subset of $X$ and $\Phi\ch$ a
finite subset of $Y$, called the \emph{roots} and \emph{coroots},
respectively, and there is a one-to-one correspondence $\ch : \Phi \rightarrow
\Phi\ch$ such that $\langle \alpha, \alpha\ch \rangle = 2$ for all $\alpha \in
\Phi$.

If $\alpha \in \Phi$ then $s_\alpha : x \mapsto x - \langle x, \alpha\ch \rangle \alpha$ is a reflection on $X$ leaving $\Phi$ invariant and $W = \langle s_\alpha \mid \alpha \in \Phi\rangle$ is a Coxeter group.
Similarly, $s_\alpha\ch: y \mapsto y -\langle \alpha, y \rangle \alpha\ch$ is a reflection on $Y$ leaving $\Phi\ch$ invariant and the group generated by all these is isomorphic to $W$.
In particular there are $\alpha_1, \ldots, \alpha_l \in \Phi$, linearly independent in $X \otimes \mathbb Q$, such that $\Phi = \Phi^+ \;\dot\cup\; \Phi^-$, where $\Phi^+ = \Phi \cap (\mathbb{N}\alpha_1 + \cdots + \mathbb{N}\alpha_l)$ and $\Phi^- = -\Phi^+$. 
The roots $\alpha_1, \ldots, \alpha_l$ and the coroots $\alpha_1\ch, \ldots,
\alpha_l\ch$ are called \emph{simple}. The number $l$ is called the
\emph{semisimple rank} of $L$ (and of $G$).

The pair $(W,S)$, where $S = \{ s_{\alpha_1}, \ldots, s_{\alpha_l} \}$, is a
Coxeter system.  The \emph{Cartan matrix} $C$ of $R$ is the $l \times l$
matrix whose $(i,j)$ entry is $\langle \alpha_i, \alpha_j\ch\rangle$.  The
matrix $C$ is related to the Coxeter type of $(W,S)$ as follows: $s_{\alpha_i}
s_{\alpha_j}$ has order $m_{ij}$ where
\[
\cos\left(\frac{\pi}{m_{ij}}\right)^2=\frac{\langle \alpha_i, \alpha_j\ch\rangle \langle \alpha_j, \alpha_i\ch\rangle}{4}.
\]
The \emph{Coxeter matrix} is $(m_{ij})_{1\le i,j\le l}$ and the \emph{Coxeter
diagram} is a graph-theoretic representation thereof: it is a graph with
vertex set $\{1,\ldots,l\}$ whose edges are the pairs $\{i,j\}$ with
$m_{ij}>2$; such an edge is labeled $m_{ij}$. The Cartan matrix $C$ 
determines the Dynkin diagram (and vice versa). For, the
\emph{Dynkin diagram} is the Coxeter diagram with the following extra
information about root lengths: $\langle \alpha_i, \alpha_j\ch \rangle <
\langle \alpha_j, \alpha_i\ch \rangle$ if and only if the Coxeter diagram edge
$\{i,j\}$ (labelled $m_{ij}$) is replaced by the directed edge $(i,j)$ in the
Dynkin diagram (so that the arrow head serves as a mnemonic for the inequality
sign indicating that the root length of $\alpha_i$ is larger than the root
length of $\alpha_j$).

A root datum is called \emph{irreducible} if its Coxeter diagram is connected.
A root datum is called \emph{semisimple} if its rank is equal to its
semisimple rank.  Each semisimple root datum can be decomposed uniquely into
irreducible root data.  The Dynkin diagrams of irreducible root systems are
well known, and described in Cartan's notation ${\rm A}_n$ $(n\ge1)$, ${\rm
B}_n$ $(n\ge2)$, ${\rm C}_n$ $(n\ge3)$, ${\rm D}_n$ $(n\ge4)$, ${\rm E}_n$
$(n\in\{6,7,8\})$, ${\rm F}_4$, ${\rm G}_2$.  The nodes are usually labeled as
in \cite{Bou81}.

For computations, we fix $X = Y = \mathbb Z^n$ and set $\langle x,y \rangle =
xy^\top$, which is an element of $\mathbb Z$ since $x$ and $y$ are row
vectors.  Now take $A$ to be the integral $l \times n$ matrix containing the
simple roots as row vectors; this matrix is called the \emph{root matrix} of
$R$. Similarly, let $B$ be the $l \times n$ matrix containing the simple
coroots; this matrix is called the \emph{coroot matrix} of $R$.  Then $C =
AB^\top$ and $\Z \Phi = \Z A$ and $\Z \Phi\ch = \Z B$.  For $\alpha \in \Phi$
we define $c^\alpha$ to be the $\Z$-valued size $l$ row vector satisfying
$\alpha = c^\alpha A$.  In the greater part of this paper, including Theorems
\ref{th:main} and \ref{thm_liealg_has_chevbas}, we will let $G$ be a simple
group, so $l=n$.

\subsection{Chevalley Lie Algebras}\label{sec:CLA}
Given a root datum $R$ we consider the free $\Z$-module 
\[
\LZ(R) = Y \oplus \bigoplus_{\alpha \in \Phi} \Z X_\alpha,
\]
where the $X_\alpha$ are formal basis elements. The rank of $\LZ(R)$ is $n +
|\Phi|$.  We denote by $[\cdot,\cdot]$ the alternating bilinear map $\LZ(R)
\times \LZ(R) \rightarrow \LZ(R)$ determined by the following rules:
\[
\begin{array}{lrclr}
\mbox{For } y, z \in Y : & [y, z] & = & 0, & \mathrm{(CB}\mathbb Z\mathrm{1)}\\
\mbox{For } y \in Y, \alpha \in \Phi : & [ X_\alpha, y ] & = & \langle \alpha, y \rangle X_\alpha , & \mathrm{(CB}\mathbb Z\mathrm{2)}\\
\mbox{For } \alpha \in \Phi : & [ X_{-\alpha}, X_\alpha ] & = & \alpha\ch , & \mathrm{(CB}\mathbb Z\mathrm{3)}\\
\mbox{For } \alpha,\beta \in \Phi, \alpha \neq \pm \beta: & [X_\alpha, X_\beta] & = & 
  \left\{ \begin{array}{ll}
  N_{\alpha,\beta} X_{\alpha+\beta} & \mbox{if } \alpha + \beta \in \Phi, \cr
  0 & \mbox{otherwise.}
  \end{array} \right. & \mathrm{(CB}\mathbb Z\mathrm{4)}
\end{array}
\]
The $N_{\alpha,\beta}$ are integral structure constants chosen to be $\pm(p_{\alpha,\beta}+1)$, 
where $p_{\alpha,\beta}$ is the biggest number such that $\alpha
-p_{\alpha,\beta}\beta$ is a root and the signs are chosen (once and for all)
so as to satisfy the
Jacobi identity.
It is easily verified that $N_{\alpha,\beta} = -N_{-\alpha,-\beta}$ and it is a well-known result 
(see for example \cite{Car72}) that  such a product exists.
$\LZ(R)$ is called a \emph{Chevalley Lie algebra}.

A basis of $\LZ(R)$ that consists of a basis of $Y$ and the formal elements
$X_\alpha$ and satisfies $\mathrm{(CB}\mathbb
Z\mathrm{1)}$--$\mathrm{(CB}\mathbb Z\mathrm{4)}$ is called a \emph{Chevalley
basis} of the Lie algebra $\LZ(R)$ with respect to the split Cartan subalgebra
$Y$ and the root datum $R$.  If no confusion is imminent we just call this a
\emph{Chevalley basis} of $\LZ(R)$.

\bigskip

For the remainder of this section, we let $\LZ(R)$ be a Chevalley Lie algebra
with root datum $R$, we fix $X = Y = \mathbb Z^n$, a basis of row vectors
$e_1, \ldots, e_n$ of $X$, and a basis of row vectors $f_1, \ldots, f_n$ of
$Y$ dual to $e_1, \ldots, e_n$ with respect to the pairing
$\langle\cdot,\cdot\rangle$.  Moreover, we let $\F$ be a field, we set $h_i =
y_i \otimes 1$, $i = 1,\ldots,n$, and $H = Y \otimes \F$.  Now tensoring
$\LZ(R)$ with $\F$ yields a Lie algebra denoted $\LF(R)$ over $\F$, and the
integral Chevalley basis relations $\mathrm{(CB}\mathbb
Z\mathrm{1)}$--$\mathrm{(CB}\mathbb Z\mathrm{4)}$ can be rephrased as:

\[
\begin{array}{lrclr}
\mbox{For } i,j \in \{1, \ldots, n \} : & [h_i, h_j] & = & 0, & \mathrm{(CB1)}\\
\mbox{For } i \in \{1, \ldots, n \}, \alpha \in \Phi : & [ X_\alpha, h_i ] & = & \langle \alpha, f_i \rangle X_\alpha , & \mathrm{(CB2)}\\
\mbox{For } \alpha \in \Phi : & [ X_{-\alpha}, X_\alpha ] & = & \sum_{i=1}^n \langle e_i, \alpha\ch \rangle h_i, & \mathrm{(CB3)} \\
\mbox{For } \alpha,\beta \in \Phi, \alpha \neq \pm \beta: & [X_\alpha, X_\beta] & = & 
  \left\{ \begin{array}{ll}
  N_{\alpha,\beta} X_{\alpha+\beta} & \mbox{if } \alpha + \beta \in \Phi, \cr
  0 & \mbox{otherwise.}
  \end{array} \right. & \mathrm{(CB4)}
\end{array}
\]

A Lie algebra is called
\emph{split} if it has a split Cartan subalgebra. The Cartan subalgebra $H$ of
each Chevalley Lie algebra $L_\F(R)$ is split.  The image of a Chevalley basis
with respect to $Y$ and $R$ in $L_{\F}(R)$ is called a
\emph{Chevalley basis of $L$ with respect to $H$ and $R$}.
The interest in Chevalley Lie algebras comes from the following result.

\begin{theorem}[Chevalley \cite{Chev58}]\label{thm_liealg_has_chevbas}
Suppose that $L$ is the Lie algebra of a split simple algebraic group $G$
over $\F$ with root datum $R = (X,\Phi,Y,\Phi\ch)$, and that $H$ is a split
Cartan subalgebra of $L$. Then $L \cong \LF(R)$ and so it has a Chevalley
basis with respect to $H$ and $R$.  Furthermore, any two split Cartan
subalgebras of $L$ are conjugate under $G$.
Finally, if $G$ is simple then $R$ is irreducible.
\end{theorem}

In light of this theorem, for the proof of Theorem
\ref{th:main}, it suffices to deal with Chevalley Lie algebras corresponding
to an irreducible root datum.

\subsection{Some difficulties}\label{sec_chevbas}
So we will deal with the construction of a Chevalley basis for a Chevalley Lie
algebra $L$ over a field $\F$, given only a split Cartan subalgebra $H$.
Algorithms for finding such an $H$ have been constructed by the first author
and Murray \cite{CM06} and, independently, Ryba \cite{Ryba07}.  These
algorithms work for $\chr(\F)$ distinct from $2$ and $3$, and partly for
$\chr(\F) = 3$.  The first algorithm has been implemented in the \Magma\
computer algebra system \cite{Magma}.  For now, we assume that we are also
given the appropriate irreducible root datum $R$, but in Section \ref{sec_conclusion} we
argue that $R$ can be found from $L$ and $H$ without much effort.  The output
of our algorithm is an ordered basis $\{ X_\alpha, h_i \mid \alpha \in \Phi, i
\in \{1, \ldots, n\} \}$ of $\LF$ (based on some ordering of the elements of
$\Phi$) satisfying (CB1)--(CB4).

For fields of characteristic distinct from $2,3$, an algorithm for finding
Chevalley bases given split Cartan subalgebras has been implemented in several
computer algebra systems, for example \Magma\ \cite{Magma} and \GAP\
\cite{GAP}. For details, see for example \cite[Section 5.11]{deGraaf00}; the
algorithm {\sc CanonicalGenerators} described there produces a Chevalley basis
only up to scalars.  The scaling, however, can be accomplished by
straightforwardly solving linear equations.

If, however, we consider Lie algebras of simple algebraic groups over a field $\F$ of characteristic $2$ or $3$, the current algorithms break down in several places.
Firstly, the root spaces (joint eigenspaces) of the split Cartan subalgebra
$H$ acting on $L$ are no longer necessarily one-dimensional. This means that
we will have to take extra measures in order to identify which vectors in
these root spaces are root elements.  This problem will be dealt with in Section \ref{sec_finding_frames}.
Secondly, we can no longer always use
root chains to compute Cartan integers $\langle\alpha,\beta^\vee\rangle$,
which are the most important piece of information for the root identification
algorithm in the general case.  
We will deal with this problem in Section \ref{sec:RootIdent}.
Thirdly, when computing the Chevalley basis
elements for non-simple roots, we cannot always obtain $X_{\alpha+\beta}$ from
(CB4) by $X_{\alpha+\beta} = \frac{1}{N_{\alpha,\beta}} [ X_\alpha, X_\beta ]$
as $N_{\alpha,\beta}$ may be a multiple of $\chr(\F)$.
This problem, however, is easily dealt with by using a different order of the roots, so we will not discuss this any further.

\subsection{Roots}
Let $p$ be zero or a prime and suppose for the remainder of this section that
$\F$ is a field of characteristic $p$.  We fix a root datum $R=(X,\Phi,Y,
\Phi\ch)$ and write $L = \LF(R)$.  We define roots and their multiplicities in
$L$ as follows.  A \emph{root of $H$ on $L$}
is the function
\[
\overline{\alpha} : h \mapsto \sum_{i=1}^n \langle \alpha, y_i \rangle t_i,
\;\;\mathrm{where}\ h = \sum_{i=1}^n y_i \otimes t_i = \sum_{i=1}^n t_i h_i,
\]
for some $\alpha \in \Phi$;
here $\langle \alpha, y_i \rangle$ is interpreted in $\Z$ (if $p=0$) or
$\Z/p\Z$ (if $p \neq 0$).  Note that this implies that $\langle \alpha, h
\rangle := \overline{\alpha}(h)$ for $h \in H$ is completely determined by the
values $\langle \alpha, y_i \rangle$, $i = 1, \ldots, n$.  We write $\Phi(L,
H)$ for the set of roots of $H$ on $L$. 

For $\overline{\alpha} \in \Phi(L,H)$ we define the \emph{root space
corresponding to} $\overline{\alpha}$ to be
\[
L_{\overline{\alpha}} = \bigcap_{i=1}^n \operatorname{Ker}(\ad_{h_i}-\overline{\alpha}(h_i)).
\]
If $\overline{\alpha} \neq 0$ for all $\alpha \in \Phi$ then $L$ is a direct sum of $H$ and its root spaces $\{ L_{\overline{\alpha}} \mid \alpha \in \Phi \}$.
If on the other hand there exists an $\alpha \in \Phi$ such that $\overline{\alpha} = 0$, then $L$ is a direct sum of $L_0 = \Cent_L(H)$ and $\{ L_{\overline{\alpha}} \mid \alpha \in \Phi, \overline{\alpha} \neq 0 \}$.

Given a root $\alpha$, we define the \emph{multiplicity} of $\alpha$ in $L$ to be the number of $\beta \in \Phi$ such that $\overline{\alpha} = \overline{\beta}$. 
Observe that if $\overline{\alpha} \neq 0$ the multiplicity of $\overline{\alpha} \in \Phi(L, H)$ is equal to $\dim(L_{\overline{\alpha}})$. 
If $\overline{\alpha} = 0$ this multiplicity is equal to $\dim(L_0)-n$.
Note that $\alpha \mapsto \overline{\alpha}$ is a surjective map $\Phi \rightarrow \Phi(L,H)$, so in what follows we abbreviate $\Phi(L,H)$ to $\overline{\Phi}$.

If each root has multiplicity $1$, there is a bijection between
$\overline{\Phi}$ and $\Phi$. Our first order of business is to decide in
which cases higher multiplicities occur.  Observe that $\overline{\alpha} = 0$
if and only if $\overline{-\alpha} = 0$ so the multiplicity of the $0$-root
space is never $1$.  If $\chr(\F) = 2$, then all nonzero multiplicities are at
least $2$ as $\overline{\alpha}$ and $\overline{-\alpha}$ coincide.  Steinberg
\cite[Sections 5.1, 7.4]{St60} studied part of the classification of Chevalley
Lie algebras $L$ for which higher multiplicities occur (the simply connected
case with Dynkin type $\An$, $\Dn$, $\En$ if $\chr(\F) = 2$) in a search for
all Lie algebras $L$ with $\Aut(L/Z(L))$ strictly larger than $G$.  In Section
\ref{sec_multdim} of this paper we prove the following proposition, which
generalizes Steinberg's result to arbitrary root data.  As the multiplicity of
a root of $H$ on the Lie algebra $L$ of a central product of split reductive
linear algebraic groups is equal to the minimum over all multiplicities of its
restrictions to summands of the corresponding central sum decomposition of
$L$, the study of multiplicities of roots can easily be reduced to the case
where $G$ is simple.

\begin{proposition}\label{prop_multdim}
Let $L$ be the Lie algebra of a split simple algebraic group over
a field $\F$ of characteristic $p$ with root datum $R$.
Then the multiplicities of the roots in
$\overline{\Phi}$ are either all $1$ or as indicated in Table
\ref{tab_multdimeigenspc}.
\end{proposition}

In Table \ref{tab_multdimeigenspc}, the Dynkin type $R$ of $L$ and the
characteristic $p$ of $\F$ are indicated by $R(p)$ in the first column.
Further details regarding the table (such as the isogeny type of $R$ appearing
as a superscript on $R(p)$) are explained in the beginning of Section
\ref{sec_multdim}. This description uniquely determines the root datum $R$ and
hence the corresponding connected algebraic group $G$ up to isomorphism; see
\cite[Chapter 9]{Springer98}.

\subsection{Computing ideals of Lie algebras}\label{sec_intro_computing_ideals}
Finding an ideal $I$ of a given Lie algebra $L$ is equivalent to finding the
submodule $I$ of the $A$-module $L$, where $A$ is the associative subalgebra
of $\End(L)$ generated by all $\ad_x$ for $x$ running over a basis of $L$.
Hence, such an ideal $I$ can be found by application of the Meat-axe algorithm
to the $A$-module $L$. We will apply the Meat-axe only to modules of bounded
dimension, so that the factor $n^6$ resulting from the occurrence of
$\dim(L)^3$ in the above-mentioned estimate for the Meat-axe running time when
$\F = \F_q$ plays no role in the asymptotic timing analysis.

\begin{algorithm}{ChevalleyBasis}{Finding a Chevalley Basis}{chevbas}
{\bf in:} \>\>\> The Lie algebra $L$ over a field $\F$ of a split reductive algebraic group, \\
\>\>\> a split Cartan subalgebra $H$ of $L$, and \\
 \>\>\> a root datum $R = (X, \Phi, Y, \Phi\ch)$. \\
{\bf out:} \>\>\> A Chevalley basis $B$ for $L$ with respect to $H$ and $R$. \\
\textbf{begin} \lnreset \\
\lnp \> \textbf{let} $E, \overline{\Phi} = $ {\sc FindRootSpaces}($L$, $H$), \\
\lnp \> \textbf{let} $\mathcal{X} = $ {\sc FindFrame}($L$, $H$, $R$, $\overline{\Phi}$, $E$), \\
\lnp \> \textbf{let} $\iota = $ {\sc IdentifyRoots}($L$, $H$, $R$, $\overline\Phi$, $\mathcal X$), \\
\lnp \> \textbf{let} $X^0, H^0 = $ {\sc ScaleToBasis}($L$, $H$, $R$, $\mathcal X$, $\iota$), \\
\lnp \> \textbf{return} $X^0, H^0$.\\
\textbf{end}
\end{algorithm}

\subsection{Algorithm \ref{alg_chevbas}}\label{sec_intro_alg_chevbas}
In the remainder of this section we give a brief overview of the inner
workings of Algorithm \ref{alg_chevbas}.  It is assumed that $L$ is isomorphic
to $L_{\F}(R)$. The {\sc FindRootSpaces} algorithm consists of simultaneous
diagonalization of $L$ with respect to $\ad_{h_1}, \ldots, \ad_{h_n}$, where
$\{h_1, \ldots, h_n\}$ is a basis of $H$.  Its output is a basis $E$ of
$H$-eigenvectors of $L$ and the set $\overline\Phi$ of roots of $H$ on $L$.
This is feasible over $\F$ because the elements are semisimple and $H$ is
split.  As $\dim(L) = O(n^2)$, these operations need time $\tO(n^6 \log(q))$
for each basis element of $H$, so the total cost is $\tO(n^7 \log(q))$
elementary operations.

The algorithm called {\sc FindFrame} is more involved, and solves the
difficulties mentioned in Section \ref{sec_chevbas} by various
methods.  The output $\mathcal X$ is a \emph{Chevalley
frame}, that is, a set of the form $\{ \F X_\alpha\mid \alpha \in
\Phi\}$, where $X_\alpha$ $(\alpha\in\Phi)$ belong to a Chevalley
basis of $L$ with respect to $H$ and $R$.  If all multiplicities are $1$ then
{\sc FindFrame} is trivial, meaning that ${\mathcal X} = \{\F x\mid x \in
E\setminus H\}$ is the required result.  The remaining cases are identified by
Proposition
\ref{prop_multdim}, and the algorithms for these cases are indicated
by [${\rm A}_2$], [$\Cent$], [$\Der$], [$\BB$] in Table
\ref{tab_multdimeigenspc} and explained in Section
\ref{sec_finding_frames}. 

In {\sc IdentifyRoots} we compute Cartan integers and use these to make the
identification $\iota$ between the root system $\Phi$ of $R$ and the Chevalley
frame $\mathcal X$ computed previously. This identification is again made on
a case-by-case basis depending on the root datum $R$.  See Section
\ref{sec:RootIdent} for details.

The algorithm ends with {\sc ScaleToBasis} where the vectors $X_\alpha$
$(\alpha\in\Phi)$ belonging to members of the Chevalley frame $\mathcal X$ are
picked in such a way that $X^0 = (X_\alpha)_{\alpha \in \Phi}$ is part of a
Chevalley basis with respect to $H$ and $R$, and a suitable basis $H^0 =
\{h_1,\ldots,h_n\}$ of $H$ is computed, so that they satisfy the Chevalley
basis multiplication rules. This step involves the solving of several systems
of linear equations, similar to the procedure explained in \cite{CM06}, which
takes time $\tO(n^8
\log(q))$.

Finally, in Section \ref{sec_conclusion}, we finish the proof of Theorem
\ref{th:main} and discuss some further problems for which our algorithm may be
of use.

\section{Multidimensional root spaces}\label{sec_multdim}

\begin{table}
\renewcommand\arraystretch{1.3}%
\[
\begin{array}{cc} 
\begin{array}{llllll}
\mbox{$R(p)$} & \mbox{\textrm{Mults}} & \mbox{\textrm{Soln}} \cr
\hline
\hline
\mathrm A_2\SC(3) & 3^2 & [\Der] \cr
\mathrm G_2(3) & 1^6, 3^2  & [\Cent]\cr
\mathrm A_3^{\mathrm{sc}, (2)}(2) & 4^3 & [\Der] \cr
\mathrm B_2\Ad(2) & 2^2, 4 & [\Cent] \cr
\mathrm B_n\Ad(2) \; (n \geq 3) & 2^n, 4^{n \choose 2} & [\Cent]\cr
\mathrm B_2\SC(2) & \mathbf{4},4 & [\mathrm B_2\SC] \cr
\mathrm B_3\SC(2) & 6^3 & [\Der] \cr
\mathrm B_4\SC(2) & 2^4,8^3 & [\Der] \cr
\mathrm B_n\SC(2) \; (n \geq 5) & 2^n, 4^{n \choose 2} & [\Cent] \cr
\end{array} &
\begin{array}{llllll}
\mbox{$R(p)$} & \mbox{\textrm{Mults}} & \mbox{\textrm{Soln}} \cr
\hline
\hline
\mathrm C_n\Ad(2) \; (n \geq 3) & 2n, 2^{n(n-1)} & [\Cent]\cr
\mathrm C_n\SC(2) \; (n \geq 3) & \mathbf{2n}, 4^{n \choose 2} & [\mathrm B_2\SC] \cr
\mathrm D_4^{(1), (n-1), (n)}(2) & 4^6 & [\Der] \cr
\mathrm D_4\SC(2) & 8^3 & [\Der] \cr
\mathrm D_n^{(1)}(2) \; (n \geq 5) & 4^{n \choose 2} & [\Der] \cr
\mathrm D_n\SC(2) \; (n \geq 5) & 4^{n \choose 2} & [\Der] \cr
\mathrm F_4(2) & 2^{12}, 8^3 & [\Cent] \cr
\mathrm G_2(2) & 4^3 & [\Der] \cr
\mbox{all remaining}(2) & 2^{|\Phi^+|} & [\mathrm A_2]\cr
\end{array}
\end{array}
\]
\renewcommand\arraystretch{1}%
\caption{Multidimensional root spaces}\label{tab_multdimeigenspc}
\end{table}

In this section we prove Proposition \ref{prop_multdim}, but first we explain the notation in Table \ref{tab_multdimeigenspc}.
As already mentioned, the first column contains the root datum $R$
specified by means of the Dynkin type
with a superscript for the isogeny type,
as well as  (between parentheses) the characteristic $p$.
A root datum of type $\mathrm A_3$ can have any of three isogeny types:
adjoint, simply connected, or an intermediate one, corresponding to the
subgroup of order $1$, $4$, and $2$ of its fundamental group $\Z/4\Z$,
respectively.  We denote the intermediate type by $\mathrm
A_3^{(2)}$.
For computations we fix root and coroot matrices for each isomorphism class of
root data, as indicated at the end 
of Section \ref{sec_intr_rootdata}. For $\mathrm A_3$, for example,
the Cartan matrix is
\[
C = \left( \begin{matrix} 2 & -1  & 0 \\ -1 & 2 & -1 \\ 0 & -1 & 2 \end{matrix} \right).
\]
As always, for the adjoint isogeny type $\mathrm A_3\Ad$ the root matrix $A$
is equal to the identity matrix $I$ and the coroot matrix $B$ is equal to
$C$. Similarly, for $\mathrm A_3\SC$ we have $A = C$ and $B = I$.  For the
intermediate case $\mathrm A_3^{(2)}$ for instance, we take
\[
A = \left( \begin{matrix} 1 & 0 & 0 \\ 0 & 1 & 0 \\ 1 & 0 & 2 \end{matrix} \right) \mbox{ and }
B = \left( \begin{matrix} 2 & -1 & -1 \\ -1 & 2 & 0 \\ 0 & -1 & 1 \end{matrix} \right).
\]
It is straightforward to check that indeed $\det(A) = 2 = \det(B)$ and $AB^\top = C$.

A root datum of type $\Dn$ has fundamental group isomorphic to
$\Z/4\Z$ if $n$ is odd, and to $(\Z/2\Z)^2$ if $n$ is even. The
unique intermediate type in the odd case is denoted by $\Dn^{(1)}$,
and the three possible intermediate types in the even case by
$\Dn^{(1)}$, $\Dn^{(n-1)}$, and $\Dn^{(n)}$.

The multiplicities appear in the second column under Mults.  Those shown in
bold correspond to the root $0$. For instance, for $\mathrm B_2\SC(2)$ we have
$\dim(\Cent_L(H)) = 6$, so the multiplicity equals $6-2 = 4$.

The third column, with header Soln, indicates the method chosen by our
algorithm. Further details appear later, in Section \ref{sec_finding_frames}.

\bigskip
Assume the setting of Proposition \ref{prop_multdim}.  By Theorem
\ref{thm_liealg_has_chevbas} there is an irreducible root datum $R = (X, \Phi,
Y, \Phi\ch)$ such that $L = \Lie(G)$ satisfies $L \cong \LF(R)$. Also, all
split Cartan subalgebras $H$ of $L$ are conjugate under $G$, so the
multiplicities of $\LF(R)$ do not depend on the choice of $H$.  For the proof
of the proposition, there is no harm in identifying $L$ with $L_\F(R)$
and $H$ with the Lie algebra of a fixed split maximal torus of $G$.

As all multiplicities are known to be $1$ if $\chr(\F)=0$, we will assume
that $p :=\chr(\F)$ is a prime.  We will write $\equiv$ for equality
$\mathrm{mod}\ p$ (to prevent confusion we will sometimes add: $\mathrm{mod}\
p$).  We begin with two lemmas.

\begin{lemma}\label{lem_nullspaceA}
Let $\alpha, \beta \in \Phi$.
Then $\overline{\alpha} = \overline{\beta}$ if and only if
$(c^\alpha - c^\beta)A \equiv 0$.
\end{lemma}

\begin{proof}
For $h\in H$, by definition, $\langle \alpha, h \rangle = \langle c^\alpha A,
h \rangle = c^\alpha A h^\top$.  This implies that $\overline{\alpha} =
\overline{\beta}$ if and only if $c^\alpha A h^\top \equiv c^\beta A h^\top$
for all $h \in H$, which is equivalent to $(c^\alpha - c^\beta)A \equiv 0$.
\end{proof}

\begin{lemma}\label{lem_multdim_isog}
Let $R_1$, $R_2$ be irreducible root data of the same rank and with the same
Cartan matrix $C$ and denote their
root matrices by $A_1$ and $A_2$, respectively.  
\begin{enumerate}[(i)]
\item
If $\det(A_2)$ strictly divides
$\det(A_1)$, then the multiplicities in $L_\F(R_1)$ are greater than or equal
to those in $L_\F(R_2)$.
\item If $p\not| \det(C)$, then the multiplicities of $L_\F(R_1)$ 
and  $L_\F(R_2)$ are the same.
\end{enumerate}
\end{lemma}
\begin{proof}
(i). Without loss of generality, we identify the ambient lattices $X$ and $Y$
with $\Z^n$ and choose the same bilinear pairing (as in Section
\ref{sec_intr_rootdata}) for each of the two root data $R_1$ and $R_2$. The
condition that $\det(A_2)$ strictly divides $\det(A_1)$ then implies that the
columns of $A_1$ belong to the lattice spanned by the columns of $A_2$. Hence
$A_1 = A_2 M$ for a certain integral $n \times n$ matrix $M$.  Thus $(c^\alpha
- c^\beta)A_2 \equiv 0$ implies $(c^\alpha - c^\beta)A_1 \equiv (c^\alpha -
c^\beta)A_2 M \equiv 0$, proving the lemma in view of Lemma
\ref{lem_nullspaceA}.

\smallskip\noindent(ii). 
As $\det(C)\not\equiv 0$, the determinants of
the coroot matrices $B_1$ and $B_2$ are nonzero modulo $p$, and
$A_1 = A_2 (B_2B_1^{-1})$ and $A_2 = A_1(B_1B_2^{-1})$.
It follows that $(c^\alpha - c^\beta)A_2 \equiv 0$
is equivalent to $(c^\alpha - c^\beta)A_1 \equiv 0$.
\end{proof}

A typical case where part (i) of this lemma can be applied is when the adjoint
and simply connected case have the same multiplicities, for then every
intermediate type will have those multiplicities as well.  It immediately
follows from Lemma \ref{lem_multdim_isog} that the root space dimensions are
biggest in the simply connected case, and least in the adjoint case.
Thus
considering root data of the adjoint and simply connected isogeny types often
suffices to understand the intermediate cases.
Part (ii) indicates that in many cases even one isogeny type will do.

The proof of Proposition \ref{prop_multdim} follows a division of cases
according to the different Dynkin types of the root datum $R$.
For each type, we need to determine when
distinct roots $\alpha, \beta$ exist in $\Phi$ 
such that $\overline{\alpha} = \overline{\beta}$.
By Lemma \ref{lem_multdim_isog}(ii), there are deviations from the adjoint
case only if $p$ divides $ \det(C)$. 

As $W$ embeds in $N_G(H)/T$, and acts equivariantly on $\Phi$ and
$\overline{\Phi} = \Phi(L,H)$, the multiplicity of a root $\overline
\alpha\in\overline{\Phi}$  only depends on the $W$-orbit of
$\alpha\in\Phi$. By transitivity of the Weyl group on roots of the same length
in $\Phi$, it suffices to consider only $\alpha = \alpha_1$ in the cases where
all roots in $\Phi$ have the same length ($\An, \Dn, \En$) and $\alpha =
\alpha_1$ or $\alpha_n$ if there are multiple root lengths ($\Bn, \Cn, \Fn,
\Gn$).

In the adjoint cases, the simple roots $\alpha_1, \ldots, \alpha_n$ are 
the standard basis vectors $e_1, \ldots, e_n$, since then the root
matrix $A$ and the coroot matrix $B$ are $I$ and $C^\top$,
respectively.  Similarly, in the simply connected cases, the simple roots
$\alpha_1, \ldots, \alpha_n$ are the rows of the Cartan matrix
$C$, since then $A=C$ and $B=I$.
We write $c = c^\beta$ so $\beta = cA$ and either all $c_i \in
\N$ or all $c_i \in -\N$.

We give the proofs of the cases where $R$ is of type $\An$, $\Bn$, or $\Gn$. The other cases are proved in a similar way.
For $V$ a linear subspace of $L$ and $x \in L$, we write $\Cent_V(x)$ for the null space of $\ad_x$ on $V$, i.e.,
\[
\Cent_V(x) := \{ v \in V \mid [x,v] = 0 \} .
\]

\subsection{\protect $\mathbf{A_n (n \geq 1)}$}
The root datum of type $\An$ has Cartan matrix 
\[
C = \left(
\begin{matrix}
2 & -1 & 0 & \ldots & 0 \cr
-1 & 2 & -1 & \ldots & 0 \cr
\vdots & \ddots & \ddots & \ddots & \vdots \cr
0 & \ldots & -1 & 2 & -1 \cr
0 & \ldots & 0 & -1 & 2 \cr
\end{matrix} \right),
\]
and the roots are
\[
\pm ( \alpha_j + \cdots + \alpha_k ), \;\;\; j \in \{ 1, \ldots, n \}, \; k \in \{j, \ldots, n \},
\]
where $\{\alpha_1,\ldots, \alpha_n\}$ are the simple roots, thus giving a total of $2 \cdot \frac{1}{2} n (n+1)$ roots.

For the adjoint case, suppose $\overline{\alpha_1} = \overline{\beta}$. 
Observe that all $c_i\in\{0,\pm1\}$.
Since $A=I$, we must have $c_1 \equiv 1$ and $c_j \equiv 0$ ($j= 2,\ldots,
n$), which implies either $p \neq 2$, $c_1 =  1$, and $c_2 = \cdots =c_n  =
0$, or $p = 2$, $c_1 = \pm 1$, and $c_2 = \cdots =c_n  = 0$. Since we assumed
$\beta \neq \alpha_1$ we find $p=2$ and $\beta = -\alpha_1$, giving
$\frac{n^2+n}{2}$ root spaces of dimension $2$.

In the simply connected case the simple roots are equal to the rows of $C$, so that $\overline{\alpha_1} = \overline{\beta}$ implies $2c_1 -c_2 \equiv 2$, $-c_1 + 2c_2 - c_3 \equiv -1$, $-c_{j-2} + 2c_{j-1} - c_j \equiv 0$ for $j = 4, \ldots, n$, and $-c_{n-1} +2c_n \equiv 0$.

We distinguish three possibilities: $c_1 = 1$, $c_1 = 0$, and $c_1 = -1$.
If $c_1 = 1$, then $c_2 \equiv 0$, so $c_2 = 0$. As $c_1 \alpha_1 + \cdots +
c_n \alpha_n$ must be a root, this implies $c_3 = \cdots = c_n = 0$, forcing
$\overline{\beta} = \overline{\alpha_1}$, a contradiction.

If $c_1 = 0$, then $-c_2 \equiv 2$, so that either $p=2$ and $c_2=0$, or $p=3$ and $c_2 = 1$. 
In the first case, we find $c_3 \equiv 1$, giving a contradiction if $n \geq
5$ (because then $c_4 \equiv 0$ and $c_5 \equiv 1$), a contradiction if $n=4$
(because then the last relation becomes $0 = -c_3 + 2c_4$, which is not
satisfied). Consequently, $n=3$ and $p = 2$; the resulting case is
discussed below. 
In the second case, where $p=3$ and $c_2 = 1$, we find $-1 \equiv 2-c_3$, so
that $c_3 \equiv 0$, giving a contradiction if $n \geq 4$ (because then $c_4
\equiv 1$), a contradiction if $n = 3$ (because then the last relation becomes
$0 = -c_2 + 2c_3$, which is not satisfied). It follows that $n=2$ and $p=3$;
this case is also discussed below.

If $c_1 = -1$, then $-c_2 \equiv 4$, so that either $p=2$ and $c_2 = 0$, or $p=3$ and $c_2 = -1$. 
In the first case, we find $c_3 = \cdots = c_n = 0$, so  $\beta = -\alpha_1$.
In the second case, we find that either $n = 2$ (the special case below), or $c_3 = 0$, which leads to a contradiction if $n \geq 4$ (because then $c_3 = 0$ but $c_4 \neq 0$), and also if $n =3$ (because then the last equation becomes $0 = -c_2 + 2c_3$).

We next determine the multiplicities in the two cases found to occur for
$\An\SC$. For $n=3$ and $p=2$ we have
\[
A=C=\left( \begin{matrix} 2 & -1 & 0 \cr -1 & 2 & -1 \cr 0 & -1 & 2\end{matrix} \right) \equiv \left( \begin{matrix} 0 & 1 & 0 \cr 1 & 0 & 1 \cr 0 & 1 & 0\end{matrix} \right) \mod 2.
\]
This gives $\overline{\alpha_1} = \overline{\alpha_3}$, as well as $\overline{
  \alpha_1 + \alpha_2} = \overline{\alpha_2 + \alpha_3}$ and
$\overline{\alpha_2} = \overline{\alpha_1 + \alpha_2 + \alpha_3}$, accounting
for $3$ root spaces of dimension $4$.

For $n=2$ and $p=3$ we have 
\[
A=C=\left( \begin{matrix} 2 & -1 \cr -1 & 2 \end{matrix} \right) \equiv \left( \begin{matrix} -1 & -1 \cr -1 & -1 \end{matrix} \right) \mod 3,
\]
which implies $\overline{\alpha_1} = \overline{\alpha_2}$ and
$\overline{\alpha_1} = \overline{-(\alpha_1 + \alpha_2)}$. Similarly,
$\overline{-\alpha_1} = \overline{-\alpha_2} = \overline{\alpha_1 +
  \alpha_2}$, giving $2$ root spaces of dimension $3$.

For the intermediate cases observe that by Lemma \ref{lem_multdim_isog}(i) we need only consider $(n,p)=(2,3)$ and $(3,2)$. 
But the former case has no intermediate isogeny types,
and the latter case is readily checked to be as stated.
This finishes the proof for the $\An$ case.

\subsection{\protect $\mathbf{B_n (n \geq 2)}$}
The root datum of type $\Bn$ has Cartan matrix 
\[
C = \left(
\begin{matrix}
2 & -1 & 0 & \ldots & 0 \cr
-1 & 2 & -1 & \ldots & 0 \cr
\vdots & \ddots & \ddots & \ddots & \vdots \cr
0 & \ldots & -1 & 2 & -2 \cr
0 & \ldots & 0 & -1 & 2 \cr
\end{matrix} \right),
\]
and the roots are
\[
\begin{array}{lll}
\phantom{\mathrm{a}} & \pm ( \alpha_j + \cdots + \alpha_l ), & j \in \{ 1, \ldots, n \}, l \in \{j, \ldots, n \}, \cr
\phantom{\mathrm{b}} & \pm ( \alpha_j + \cdots + \alpha_{l-1} + 2\alpha_l + \cdots + 2\alpha_n ), & j \in \{ 1, \ldots, n-1 \}, l \in \{j+1, \ldots, n\},
\end{array}
\]
giving a total of $2\cdot \frac{1}{2}n(n+1) + 2\cdot \frac{1}{2}n(n-1) = 2n^2$ roots.

In the adjoint case we have $A=I$. For the long roots, suppose
$\overline{\alpha_1} = \overline{\beta}$, so $c_1 \equiv 1$ and $c_2 \equiv
\cdots \equiv c_n \equiv 0$.  If $c_1 = 1$, then $c_2 \neq 0$ (for otherwise
$\beta=\alpha_1$), which implies $p=2$ and $\beta = \alpha_1 + 2\alpha_2 +
\cdots +2\alpha_n$.  If $c_1 = -1$, then $p=2$, and either $c_2 = 0$, which
gives $\beta = -\alpha_1$, or $c_2 \neq 0$, which implies $\beta = -\alpha_1
-2\alpha_2 - \cdots -2\alpha_n$.  In this case the long roots have
multiplicities $4$.

In the adjoint case, for the short roots, suppose $\overline{\alpha_n} = \overline{\beta}$, so $c_n \equiv 1$ and $c_1 \equiv \cdots \equiv c_{n-1} \equiv 0$. 
This yields three possibilities for $c_n$:
If $c_n = -2$, then $p=3$, implying $c_{n-1}$ is either $0$ or $-3$, neither of which give rise to roots.
If $c_n = -1$, then $p=2$; now either $c_{n-1} = 0$ (yielding $\beta = -\alpha_n$), or $c_{n-1} = -2$ (not giving any roots).
If $c_n = 1$ we must have $c_{n-1} = \cdots = c_1 = 0$, giving
the contradiction $\beta = \alpha_n$.
This shows that $p=2$ and all multiplicities are $2$.

\bigskip
In the simply connected case we have $A=C$. 
We will first consider $n\geq 5$, and then treat $n=2,3,4$ separately.
By Lemma \ref{lem_multdim_isog}(ii), we may assume $p=2$.

For the long roots, suppose $\overline{\alpha_1} =\overline{\beta}$, so $c_2
\equiv 0$, $c_1 +c_3 \equiv 1$, and $c_{j-2}+c_j \equiv 0$ ($j = 4, \ldots,
n)$.  This forces $c_4 \equiv 0$. If $c_1 \equiv 0$ then $c_1=0$ and hence
$c_2=0$, so $c_3 =\pm 1$. replacing $\beta$ by $\beta$ if needed, we may
assume $c_3 = 1$. As $c_4 \equiv 0$ and $c_5 \equiv 1$, we must have $c_4=2$
and $c+5=1$, which is never satisfied by a root.  If on the other hand $c_1
\equiv 1$ then $c_3 \equiv c_4 \equiv \cdots \equiv c_n \equiv 0$, so $\beta =
-\alpha_1$ or $\beta = \pm(\alpha_1 + 2\alpha_2 + \cdots + 2\alpha_n)$.  This
shows that, for $n \geq 5$, the multiplicities of $\overline\beta$ for $\beta$
a long root are $4$.

For the short roots, suppose $\overline{\alpha_n} = \overline{\beta}$, so $c_2
\equiv 0$, $c_{j-2} + c_j \equiv 0$ ($j = 3, \ldots, n-1$), and $c_{n-2}+c_n
\equiv 1$.  If $c_1 \equiv 1$ then $c_3 \equiv 1$, but since $c_2 \equiv 0$
this contradicts that $\beta$ is a root.  If on the other
hand $c_1 \equiv 0$, then $c_2 \equiv c_3 \equiv \cdots \equiv c_{n-1} \equiv
0$, so $c_n \equiv 1$ and we
find $\beta = -\alpha_n$. 
Hence, for $n \geq 5$, the multiplicities of $\overline\beta$ for $\beta$
a short root are $2$.

\bigskip

If $n=2$ then
\[
C = \left( \begin{matrix} 2 & -2 \\ -1 & 2 \end{matrix} \right) \equiv
\left( \begin{matrix} 0 & 0 \\ 1 & 0 \end{matrix} \right)
\]
If $\overline{\alpha_1} = \overline{\beta}$ we have $c_2 \equiv 0$.
Since $-2 \leq c_2 \leq 2$ we must have either $c_2 = 0$ (hence $\beta = -\alpha_1$), or $c_2 = \pm 2$ (hence $c_1 = \pm 1$), giving $\beta = \pm\alpha_1$ or $\beta = \pm (\alpha_1 + 2\alpha_2)$.
If on the other hand $\overline{\alpha_2} = \overline{\beta}$ we find $c_2 \equiv 1$ hence $\beta = \pm \alpha_2$ or $\beta = \pm(\alpha_1 + \alpha_2)$. 
This shows that $B_2\SC$ has $2$ root spaces of dimension $4$ if $p=2$.

\bigskip

If $n=3$ then
\[
C = \left( \begin{matrix} 2 & -1 & 0 \\ -1 & 2 & -2 \\ 0 & -1 & 2 \end{matrix} \right) \equiv 
 \left( \begin{matrix} 0 & 1 & 0 \\ 1 & 0 & 0 \\ 0 & 1 & 0 \end{matrix} \right) 
\]
{From} a straightforward case distinction on the roots of $\mathrm B_3$
and the fact that $\overline{\alpha_1} = \overline{\alpha_3}$ we immediately see that
$\overline{\alpha_1} = \overline{\alpha_3} = \overline{\alpha_1 + 2\alpha_2 + 2\alpha_3}$,  
$\overline{\alpha_2} = \overline{\alpha_1 + \alpha_2 + \alpha_3} = \overline{\alpha_2 + 2 \alpha_3}$, and
$\overline{\alpha_1 + \alpha_2} = \overline{\alpha_2 + \alpha_3} = \overline{\alpha_1 + \alpha_2 + 2\alpha_3}$. 
This gives the $3$ required root spaces of dimension $6$.

\bigskip

If $n=4$ then
\[
C = \left( \begin{matrix} 2 & -1 & 0 & 0 \\ -1 & 2 & -1 & 0 \\ 0 & -1 & 2 & -2 \\ 0 & 0 & -1 & 2 \end{matrix} \right) \equiv \left( \begin{matrix} 0 & 1 & 0 & 0 \\ 1 & 0 & 1 & 0 \\ 0 & 1 & 0 & 0 \\ 0 & 0 & 1 & 0 \end{matrix} \right).
\]
{From} a straightforward case distinction on the roots of $\mathrm B_4$
and the fact that $\overline{\alpha_1} = \overline{\alpha_3}$, we find
$\overline{\alpha_1} = \overline{\alpha_3} = 
 \overline{\alpha_3 + 2\alpha_4} =
 \overline{\alpha_1 + 2\alpha_2 + 2\alpha_3 + 2\alpha_4}$, as well as
$\overline{\alpha_2} = \overline{\alpha_1 + \alpha_2 + \alpha_3} = 
 \overline{\alpha_1 + \alpha_2 + \alpha_3 + 2\alpha_4} =
 \overline{\alpha_2 + 2\alpha_3 + 2\alpha_4}$ and
$\overline{\alpha_1+\alpha_2} = 
 \overline{\alpha_2 + \alpha_3} = 
 \overline{\alpha_2 + 2\alpha_3 + 2\alpha_4} =
 \overline{\alpha_1 + \alpha_2 + 2\alpha_3 + 2\alpha_4}$.
The remaining $32-24=8$ roots ($\pm(\alpha_j + \cdots + \alpha_n), j=1,\ldots,4$) are in $2$-dimensional spaces, giving $2^4, 8^3$, as required.

\subsection{\protect $\mathbf{\Gn}$}
The root datum of type $\Gn$ has Cartan matrix 
\[
C = \left(
\begin{matrix}
2 & -1 \cr
-3 & 2 
\end{matrix} \right),
\]
and the roots are
\[
\begin{array}{ll}
\pm \alpha_1, \pm (\alpha_1 + \alpha_2), \pm( 2\alpha_1 + \alpha_2), & \mbox{($6$ short roots)} \\
\pm \alpha_2, \pm( 3\alpha_1 + \alpha_2),\pm( 3\alpha_1 + 2\alpha_2),& \mbox{($6$ long roots)} 
\end{array}
\]
giving a total of $12$ roots.  As $\det(C) = 1$, we take $A=I$.  All
components of $c$ are in $\{-3, \ldots, 3\}$, so all components of
the differences $\alpha_1
-\beta$ and $\alpha_2 -\beta$ are in $\{-4,\ldots,4\}$.
Hence, if multidimensional root spaces occur, we must have $p \leq 3$.

If $p=3$ we see $\overline{3\alpha_1+\alpha_2} = \overline{\alpha_2} =
\overline{-(3\alpha_1 + 2\alpha_2)}$ and $\overline{-(3\alpha_1+\alpha_2)} =
\overline{-\alpha_2} = \overline{3\alpha_1 + 2\alpha_2}$, and the remaining
$6$ roots all have distinct root spaces. 

If $p=2$ we find $\overline{\alpha_1+\alpha_2} =
\overline{3\alpha_1+\alpha_2}$, $\overline{\alpha_1} =
\overline{3\alpha_1+2\alpha_2}$ and $\overline{\alpha_2} =
\overline{2\alpha_1+\alpha_2}$, giving $3$ root spaces of dimension $4$.

This finishes the proof of Proposition \ref{prop_multdim}.

\section{Finding Frames}\label{sec_finding_frames}
Let $L$ be a Chevalley Lie algebra over $\F$ with root datum $R$, a fixed
split Cartan subalgebra $H$, and given decomposition $E$ into root spaces with
respect to the set $\overline{\Phi}=\Phi(L,H)$ of roots of $H$ on $L$.  In
this section we discuss the procedure of Algorithm \ref{alg_chevbas} referred
to as {\sc FindFrame}.  It determines the set $ \mathcal{X} = \{ \F X_\alpha
\mid \alpha \in \Phi\}$, i.e., the one-dimensional root spaces with respect to
$\Phi$, to which we refer as the \emph{Chevalley frame}.  Note that we do
not yet \emph{identify} the root spaces: finding a suitable bijection between
$\Phi$ and the Chevalley frame $\mathcal X$ is discussed in the next section.
We set $p=\chr(\F)$.

We require that $R$ be given, since we execute different
algorithms depending on $R$, for example $\mathrm B_2\Ad$ needs $[\Cent]$
whereas $\mathrm B_2\SC$ needs $[\mathrm B_2\SC]$.

For $p=2$, we use the procedure described in Section
\ref{sec_multdimsol_A2} to find the frame once we have computed all spaces $\F
X_\alpha + \F X_{-\alpha}$ for $\alpha \in \Phi$.  To this algorithm we will
refer as $[\mathrm A_2]$. As an auxiliary result, this procedure stores the
unordered pairs $\{ \{\alpha,-\alpha\} \mid \alpha \in \Phi^+\}$, to be used in the {\sc
IdentifyRoots} procedure discussed in Section \ref{sec:RootIdent}
(notably, the proof of Lemma \ref{lem_can_compute_cartint1}).

The general method in characteristic $2$ is to reduce the root spaces of dimension
greater than $2$ to such $2$-dimensional spaces, and apply $[\mathrm A_2]$.
For this purpose, and for the two cases of characteristic $3$, we distinguish
three general methods:
\begin{itemize}
\item $[\Cent]$: Given two root spaces $M,M'$ compute $C_M(M')$ to break down $M$. Often, but not always, $\dim(M')=2$. An example of this method is given in Section \ref{sec_multdimsol_G23}.

\item $[\Der]$: Compute the Lie algebra $\Der(L)$ of derivations of $L$, and
calculate in there. This is a useful approach if $\Der(L)$ is strictly larger
than $L$, for then we can often extend $H$ to a larger split Cartan
subalgebra, so we find new semisimple elements acting on the root spaces. Examples of
this method are given in Sections \ref{sec_multdimsol_D4} and
\ref{sec_multdimsol_G22}.

\item $[\mathrm B_2\SC]$: The case where $R(p) = \mathrm B_2\SC(2)$
is slightly more involved than the other cases because $\overline{\alpha}=0$
for some $\alpha \in \Phi$. We use the Meat-axe to split the action of the
long roots on the short roots. Examples of this method are given in Sections
\ref{sec_multdimsol_B2SC} and \ref{sec_multdimsol_CnSC}.
\end{itemize}
The method chosen depends on the root datum $R$ and the characteristic $p$, as
indicated in the third column of Table \ref{tab_multdimeigenspc}.

\subsection{\protect $\mathbf{A_2}\mbox{ in characteristic } 2$}\label{sec_multdimsol_A2}
First, we consider the Lie algebras $L$ with $R(p) = \mathrm A_2(2)$, as
this procedure is used inside various other cases.  The isogeny type of the
root datum is of no importance here.  For clarity, we write $\alpha,\beta$ for
the two simple roots of the root system of type $\mathrm{A}_2$.

As indicated in Table \ref{tab_multdimeigenspc}, we have $3$ root spaces of
dimension $2$.  They correspond to $\langle X_\gamma,X_{-\gamma}\rangle_\F$
for $\gamma \in \{ \alpha, \beta, \alpha+\beta \}$.  Without loss of
generality we consider $L_{\overline{\alpha}} = \langle
X_\alpha,X_{-\alpha}\rangle_\F$ and $L_{\overline{\beta}} = \langle
X_\beta,X_{-\beta}\rangle_\F$.  Observe that the squared adjoint action
$\ad_{X_\alpha}^2$ of $X_\alpha$ sends any element of $L_{\overline{\beta}}$
to zero: $[X_\alpha, [X_\alpha,X_\beta]] = [X_\alpha,N_{\alpha,\beta}
X_{\alpha+\beta}] = 0$ as $2\alpha+\beta \not\in \Phi$, and $[X_\alpha,
X_{-\beta}] = 0$ since $\alpha-\beta \not\in \Phi$. Similarly,
$\ad_{X_{-\alpha}}^2(L_{\overline{\beta}})=0$.

However, the quadratic action $\ad_x^2$ of a general element $x = t_1X_\alpha+t_2 X_{-\alpha}$ ($t_1, t_2 \in \F$, both nonzero) of $L_{\overline{\alpha}}$ does not centralise $L_{\overline{\beta}}$. Indeed:
\begin{eqnarray*}
[x,[x,X_\beta]] & = & t_1 t_2\left( [X_{-\alpha}, [X_\alpha, X_\beta]] +  [X_\alpha, [X_{-\alpha}, X_\beta]] \right) \cr
& = & t_1 t_2 N_{-\alpha,\alpha+\beta} N_{\alpha,\beta} X_\beta,
\end{eqnarray*}
which is nonzero since $N_{-\alpha,\alpha+\beta}$ and $N_{\alpha,\beta}$ are
both equal to $1$ modulo $2$.

Recall that we are given $L_{\overline\alpha}$ and $L_{\overline\beta}$.  Fix
a basis $r_1, r_2$ of $L_{\overline{\alpha}}$ and consider the element $x =
r_1 + t r_2$, where $t \in \F$.  It follows from the above observations that
$\ad_x^2(L_{\overline{\beta}})=0$ if and only if $x$ is a scalar multiple of
$X_\alpha$ or $X_{-\alpha}$, so in order to find the frame elements among the
$\F x$  for $t\in \F$ we
have to solve
\begin{eqnarray*}\label{eqn_quadr_A2}
0 & = & [x,[x,y]] 
   =  [r_1 + t r_2, [r_1 + t r_2, y ]]  \cr
  & = & [r_1, [r_1, y]] + t\left( [r_1, [r_2,y]] + [r_2, [r_1,y]] \right) + t^2[r_2,[r_2,y]],
\end{eqnarray*}
for every $y \in L_{\overline{\beta}}$ in the unknown $t$.  We know there is a
solution as $H$ is split.
Solving this system is equivalent to solving
a system of $2 \cdot 3 = 6$ quadratic equations in $t$ (note that the
$[r_i,[r_j,y]]$ are in $\langle L_{\overline{\beta}} \rangle_L$, which is at
most $3$-dimensional).
If $\F=\F_q$, solving such a quadratic equation is equivalent to solving
$\log(q)$ equations in $\log(q)$ variables over $\F_2$ (as $p=2$ is fixed),
requiring $\tO(\log(q)^3)$ arithmetic operations, or
$\tO(\log(q)^4)$ elementary operations.

For more general Lie algebras $L$, the solutions for Lie subalgebras of type
${\mathrm A}_2$ normalized by $H$ will be part of a Chevalley frame. These
parts can be found inside any two-dimensional root space $V\in E$ provided
there is at least one other two-dimensional root space $V'\in E$ such that
$\langle V, V' \rangle_L$ is of type $\mathrm A_2$.  So, if all root spaces in
$E$ are 2-dimensional and $\F = \F_q$, this method needs $O(n^2)$ root spaces
$V$ to be analysed (at a cost of $\tO(n^8\log(q)^4)$ each), so that $\mathcal X$
will be found in $\tO(n^{10}\log(q)^4)$ elementary operations.

\subsection{\protect $\mathbf{G_2}\mbox{ in characteristic } 3$}\label{sec_multdimsol_G23}
Secondly, we consider the Lie algebra $L = \LF(\mathrm G_2)$ of the root datum
of type $\Gn$ over a field $\F$ of characteristic $3$.  
By Proposition \ref{prop_multdim}
there are $8$ root spaces.  It is readily verified that
$\dim(L_{\overline\alpha}) = 1$ if $\alpha$ is a short root and
$\dim(L_{\overline\alpha}) = 3$ if $\alpha$ is a long root of $\Phi$.  In
particular, the short root spaces belong to $\mathcal X$ and it remains to
split the two long root spaces.

Consider one of the two three-dimensional root spaces in $E$, say $V = \F
X_{\alpha_2} + \F X_{3 \alpha_1 + \alpha_2} + \F X_{-3\alpha_1 -
2\alpha_2}$. The left multiplications on $V$ by the short roots are easily
obtained from (CB1)--(CB4); these are given in Table \ref{tab_multG2}.

\begin{table}
\begin{equation}
\begin{array}{l|ccc}
& X_{\alpha_2} & X_{3\alpha_1 + \alpha_2} & X_{-3\alpha_1 - 2\alpha_2} \\
\hline
X_{\alpha_1} & X_{\alpha_1 + \alpha_2} & 0 & 0 \\
X_{-\alpha_1} & 0 & X_{2\alpha_1 + \alpha_2} & 0 \\
X_{\alpha_1 + \alpha_2} & 0 & 0 & X_{-2\alpha_1 - \alpha_2} \\
X_{-\alpha_1 - \alpha_2} & -X_{\alpha_1} & 0 & 0 \\
X_{2\alpha_1 + \alpha_2} & 0 & 0 & -X_{-\alpha_1} \\
X_{-2\alpha_1 - \alpha_2} & 0 & -X_{\alpha_1} & 0 
\end{array}
\end{equation}
\caption{Part of the $\Gn$ multiplication table}\label{tab_multG2}
\end{table}

Although we have not yet identified the roots, we can identify the three pairs
of one-dimensional root spaces $\{\F X_{\alpha}, \F X_{-\alpha} \}$, for
$\alpha\in\Phi$ short, since $L_{\overline{-\alpha}}$ is the unique
one-dimensional root space with root $-\overline{\alpha}$.
{From} this observation and Table \ref{tab_multG2} it follows that we can
obtain the triple $\F X_\beta$ ($\beta \in \{ \alpha_2, 3\alpha_1+\alpha_2, -3\alpha_1 +
2\alpha_2\}$) as follows:
\begin{eqnarray*}
\F X_{\alpha_2} &=& \Cent_V(L_{\overline{2\alpha_1+\alpha_2}}+L_{\overline{-2\alpha_1-\alpha_2}}),\\
\F X_{3\alpha_1+\alpha_2} &=& \Cent_V(L_{\overline{\alpha_1+\alpha_2}}+ L_{\overline{-\alpha_1-\alpha_2}}),
\\
\F X_{-3\alpha_1-2\alpha_2} &=& \Cent_V(L_{\overline{\alpha_1}}+ L_{\overline{-\alpha_1}}).
\end{eqnarray*}
For the other three-dimensional space, the same approach is used.  This
completes the search for the Chevalley frame $\mathcal X$.

\subsection{\protect $\mathbf{D_4}\mbox{ in characteristic } 2$}\label{sec_multdimsol_D4}
Thirdly, we consider the Lie algebras with Dynkin diagram
of type $\mathrm D_4$ over a field $\F$ of characteristic $2$. 
As mentioned in Section \ref{sec_multdim}, there are three cases: 
\begin{itemize}
\item[]$L\Ad$:  the adjoint root datum ($12$ two-dimensional root spaces), 
\item[]$L\SC$: the simply connected root datum ($3$ eight-dimensional root spaces), 
\item[]$L^\mathrm{(1)},L^\mathrm{(3)}$, $L^\mathrm{(4)}$: the intermediate root data ($6$ four-dimensional root spaces).
\end{itemize}
\noindent
The three intermediate root data all give rise to the same Lie algebra up to
isomorphism (by triality), so we will restrict ourselves to the study of
$L\Ad$, $L\SC$, and $L^\mathrm{(1)}$.  It is straightforward to verify that
$L\Ad$ has a $26$-dimensional ideal $I\Ad$, linearly spanned by $X_\alpha$
($\alpha \in \Phi$), $(\alpha_1\ch + \alpha_3\ch + \alpha_4\ch) \otimes 1$,
and $\alpha_2\ch \otimes 1$.
This ideal can be found, for
example, by use of the Meat-axe.

Similarly, $L\SC$ has a $2$-dimensional ideal $I$ (spanned by $(\alpha_1\ch +
\alpha_4\ch) \otimes 1$ and $(\alpha_3\ch + \alpha_4\ch) \otimes 1$). Let
$I\SC = L\SC/I$ be the $26$-dimensional Lie algebra obtained by computing in
$L\SC$ modulo $I$.  Finally, $L^\mathrm{(1)}$ has a $1$-dimensional ideal $I$
(spanned by $\alpha_4 \otimes 1$), and a $27$-dimensional ideal $I'$ (spanned
by $\alpha_4 \otimes 1$ and $X_\alpha$, $\alpha \in \Phi$). We let $I\rint =
I'/I$.  Again, the 26-dimensional ideal is easily found by means of the
Meat-axe.

Thus we have constructed three $26$-dimensional Lie algebras: $I\Ad$, $I\SC$,
and $I\rint$.  By results of Chevalley (cf.~\cite[Part 2,
Cor.~2.7]{Jantzen}) they are isomorphic, so from now on we let $I$ be one
of these $26$-dimensional Lie algebras. The Lie algebra $I$ is simple.  Its
derivation algebra $ \Der(I)$ is a Lie algebra of type $\Fn$, and thus has
$12$ two-dimensional root spaces and $3$ eight-dimensional root spaces.

Using a procedure similar to the one for $\Gn$ over characteristic $3$
described in Section \ref{sec_multdimsol_G23}, we can break up the
eight-dimensional spaces of $E$ into two-dimensional spaces, giving us $24$
two-dimensional spaces.  These two-dimensional spaces may then be broken up
into one-dimensional spaces by the procedure [$\mathrm A_2$].  The last step
in the process is ``pulling back'' the relevant one-dimensional spaces from
$\Der(I)$ to $I$. But this is straightforward, since $I$ is an ideal of
$\Der(I)$ by construction. 

\subsection{\protect $\mathbf{G_2}\mbox{ in characteristic } 2$}\label{sec_multdimsol_G22}
\newcommand{\AAA}{{\ensuremath{\mathrm{A}_3}}}
\newcommand{\GG}{{\ensuremath{\mathrm{G}_2}}}
As noted in \cite[Section 2.6]{St60},
in the exceptional case $R(p) = \Gn(2)$,
the Lie algebra $L$ 
is isomorphic to the unique 14-dimensional ideal 
of the
Chevalley Lie algebra $L^{\mathrm A}$
of adjoint type $\mathrm A_3$ over $\F$.

In particular, $\Der(L)$ contains a copy of $L^{\mathrm A}$.  We use this fact
by finding a split Cartan subalgebra $H'$ inside $\Cent_{\Der(L)}(H)$ so
that $H \subset H'$. For then
we can calculate the Chevalley frame $\mathcal X^{\mathrm A}$
inside the Lie subalgebra $\langle L, H'
\rangle_{\Der(L)}$ of $\Der(L)$ with respect to $H'$, 
which is of type $\AAA$ by the above observation.

The Chevalley frame $\mathcal X$ of $L$ is now simply the part of $\mathcal
X^{\mathrm A}$ that lies inside $L$.

\subsection{\protect $\mathbf{B_2\SC}\mbox{ in characteristic } 2$}\label{sec_multdimsol_B2SC}
We consider the Chevalley Lie algebra $L$ of type $\mathrm B_2\SC$ over a
field $\F$ of characteristic $2$ with split Cartan subalgebra $H = \F h_1 +\F
h_2$.  This is a particularly difficult case, as the automorphism group of $L$
is quite big: $\Aut(L) = G \ltimes (\F^+)^4$ \cite[Theorem 14.1]{Hogeweij78},
where $G$ is the Chevalley group of adjoint type $\mathrm B_2$ over $\F$ and
$\F^+$ refers to the additive group of $\F$.  As a consequence, there is more
choice in finding the frame than in the previous cases.

To begin, we take $L_0$ to be the $(0,0)$-root space of $H$ on $L$, and $L_1$
to be the $(1,0)$-root space of $H$ on $L$.  It is easily verified
that $L_0 = \langle H, X_{\pm \alpha_1}, X_{\pm (\alpha_1+2\alpha_2)}
\rangle_\F$ (that is, the linear span of $H$ and the long root elements) and
$L_1 = \langle X_{\pm \alpha_2}, X_{\pm(\alpha_1+\alpha_2)} \rangle_\F$ (the
linear span of the short root elements).
We proceed in three steps.

\nl $[\mathbf{B_2\SC.1}]$. The subalgebra $L_0$ has Dynkin type $\mathrm A_1
\oplus \mathrm A_1$. We may split it (non-uniquely) into two subalgebras of
type $\mathrm A_1$ using a direct sum decomposition procedure.  This is a
procedure that can be carried out with standard linear algebra arithmetic for
a fixed dimension ($6$, in this case); see e.g., \cite[Section
1.15]{deGraaf00}.

\nl
$[\mathbf{B_2\SC.2}]$. Let $A$ be one of these subalgebras of $L_0$
of type $\mathrm
A_1$.  Assume for the sake of reasoning that $A = \langle X_{\pm \alpha_1}
\rangle_L$, the Lie subalgebra of $L$ generated by $X_{\alpha_1}$ and $X_{-
\alpha_1}$.  Since $[A,L_1]=L_1$ we may view $L_1$ as a four-dimensional
$A$-module, and hence apply the Meat-axe \cite{Holt,HoltEickOBrien} to find a proper
irreducible $A$-submodule $M$ of $L_1$.  This will be a submodule of the form
\[
M = \langle t_1 X_{\alpha_2} + t_2 X_{-\alpha_1-\alpha_2}, t_1 X_{\alpha_1+\alpha_2} + t_2 X_{-\alpha_2} \rangle_\F, \quad t_1, t_2 \in \F.
\]
We take $b_1, b_2$ to be a basis of $M$, and add
$ \Cent_A(b_2)$ and 
$ \Cent_A(b_1)$ to $\mathcal X$.
These two spaces are indeed one-dimensional
and coincide with the original $\F X_{\pm \alpha_1}$ if 
$b_1 \in \F (t_1 X_{\alpha_2} + t_2 X_{-\alpha_1-\alpha_2})$ and 
$b_2 \in \F(t_1 X_{\alpha_1+\alpha_2} + t_2 X_{-\alpha_2})$.
This exhibits part of the freedom of choice induced by the factor $(\F^+)^4$ in $\Aut(L)$.

We repeat this procedure for both subalgebras of type $\mathrm A_1$ found in
the first step. The result is the part of the Chevalley frame $\mathcal X$
inside $L_0$. In fact, due to our method, we can make an identification of the long
roots $\pm\alpha_1$, $\pm(\alpha_1+2\alpha_2)$ with the four elements of
$\mathcal X$ found. In what follows we will work with such a choice so that we
have the elements $\F X_{\alpha_1}$, $\F X_{-\alpha_1}$, $\F
X_{\alpha_1+2\alpha_2}$, $\F X_{-\alpha_1-2\alpha_2}$ in $\mathcal X$ as well
as the correspondence with the roots in $\Phi$ suggested by the subscripts.

\nl
$[\mathbf{B_2\SC.3}]$. We find the part of $\mathcal X$ inside $L_1$ as follows. 
$\F X_{\alpha_1+\alpha_2}$ coincides with
$\Cent_{L_1}(\F X_{\alpha_1}, \F X_{\alpha_1+2\alpha_2})$. 
Having computed this element of $\mathcal X$, we finish by taking
\begin{eqnarray*}
\F X_{\alpha_2} &=& [\F X_{\alpha_1+\alpha_2},\F X_{-\alpha_1}],\\
\F X_{-\alpha_1-\alpha_2} &=& [\F X_{\alpha_2},\F X_{-\alpha_1-2\alpha_2}],\\
\F X_{-\alpha_2} &=& [\F X_{\alpha_1-\alpha_2},\F X_{\alpha_1}].
\end{eqnarray*}
This completes the search for $\mathcal X$ in the case $\mathrm B_2\SC(2)$
and establishes that its running time is $\tO(\log(q))$.

\subsection{\protect $\mathbf{C_n\SC}\mbox{ in characteristic } 2$}\label{sec_multdimsol_CnSC}
We consider the Chevalley Lie algebra $L$ of type $\mathrm C_n \SC$ over a
field $\F$ of characteristic $2$. Here $n\ge3$, so that the multiplicity of
$\overline 0$ is strictly larger than $4$.  Let
$h_z$ be a basis of the $1$-dimensional centre of $L$, inside the split Cartan
subalgebra $H$ of $L$.  This case is a generalisation of the $\mathrm B_2\SC$
case described in Section \ref{sec_multdimsol_B2SC}.  We again take $L_0$ to
be the $0$-root space of $H$ on $L$, so that $L_0$ is $3n$-dimensional and
consists of $H$ and the root spaces corresponding to the long roots.  Similar
to the previous case, $L_0 \cong \mathrm A_1 \oplus \cdots \oplus \mathrm A_1$
($n$ constituents), and again the decomposition is not unique. We describe how
to find such a decomposition.

\newcommand{\fourspcs}{\mathcal F} We let $\fourspcs$ be the set of ${n
\choose 2}$ four-dimensional root spaces (cf.~Table
\ref{tab_multdimeigenspc}).  In the root system of type $\mathrm C_n$ each of
these corresponds to the four roots $\pm \epsilon_i \pm \epsilon_j$ for some
$i,j \in \{1,\ldots, n\}$ with $i \neq j$.  Our first task is to split $L_0$
into subalgebras of type $\mathrm A_1$ in a way compatible with $\fourspcs$.
To this end, we let $\Gamma$ be the graph with vertex set $\fourspcs$, and
edges $f \sim g$ whenever $f \neq g$ and $[f,g] \neq 0$.

Let $\Delta$ be a maximal coclique of $\Gamma$ of size $n-1$, so that $\Delta$
consists of $n-1$ elements of $\fourspcs$ such that $[f,g] = 0$ for all $f,g
\in \Delta$.  This means that, for a particular $i \in \{1, \ldots, n\}$, the
set $\Delta \subseteq \fourspcs$ corresponds to those four-spaces in $\mathcal
F$ that arise from the roots $\pm \epsilon_i \pm \epsilon_j$, where $j \in
\{1,\ldots,n\}\setminus\{i\}$.  Let $\overline{\Delta} = \Gamma - \Delta$, so
that $\overline{\Delta}$ contains precisely the four-dimensional spaces
corresponding to $\pm \epsilon_k \pm \epsilon_l$ with $k,l \neq i$.

Now compute the centralizer $A$ in $L_0$ of all spaces in
$\overline{\Delta}$. Then $A$ coincides with $\langle X_{\pm \gamma},
\gamma\ch \otimes 1, h_z \rangle_\F$ for the long root $\gamma = 2
\epsilon_i$.  Using a direct sum decomposition procedure we find the Lie
subalgebra $A'$ of $A$ such that $A = A' \oplus \F h_z$, where $A' = \langle
X_{\pm \gamma}, \gamma\ch \otimes 1 \rangle_\F$.  The subalgebra $A'$ is one
of the type $\mathrm A_1$ constituents of $L_0$ we are after.  Thus, by
repeating this procedure for each maximal coclique of $\Gamma$ of size $n-1$,
we obtain a decomposition of $L_0$ into $n$ subalgebras of type $\mathrm A_1$.
We will denote by $\mathcal A$ the set of these $n$ subalgebras.


Now we continue as in the $\mathrm B_2\SC$ case: For each element of $\mathcal
A$ we use the procedure labelled $[\mathrm B_2\SC.2]$ to find suitable
elements $\F X_{\pm \gamma}$ for $\mathcal X$.  For each four-dimensional
space $K \in \fourspcs$ we then use distinct $S_1, S_2 \in \mathcal A$ satisfying
$[K,S_1] \neq 0$, $[K,S_2] \neq 0$ and these $\F X_{\pm \gamma}$ to execute a
$[\mathrm B_2\SC.3]$ procedure. Thus, we find the part of the frame inside $K$.

If $n=3$ splitting $L_0$ has to be done in a slightly different way, but as
this is only a slight modification of the algorithm we will not go into
details here.  This completes the Chevalley frame finding in the case $\mathrm
C_n\SC(2)$. 
Its running time involves $O(n^2)$ executions of parts of
the algorithm of Section \ref{sec_multdimsol_B2SC}, which is however dominated
by the time $\tO(n^{\maxexp}\log(q)^4)$ needed for method $[\mathrm A_2]$.

\medskip
We summarise the results of this section.
\begin{proposition}\label{prop_FindFrame}
Given $L$, $H$, $R$, the set $\overline\Phi$ of roots of $H$ on $L$,
and the root spaces $E$, the Las Vegas procedure
{\sc FindFrame} finds a Chevalley frame.
For $\F = \F_q$, it runs in time $\tO(n^{\maxexp}\log(q)^4)$.
\end{proposition}

\begin{proof}
As mentioned in Section \ref{sec_intro_alg_chevbas} this procedure is trivial
in all cases except those mentioned in Table \ref{tab_multdimeigenspc}, and
for each of the cases in Table \ref{tab_multdimeigenspc} we have presented a
solution. Recall that $|\Phi|\le \dim(L) = O(n^2)$.  

The timing of Method $[\mathrm A_2]$ is dealt with in 
Section \ref{sec_multdimsol_A2}, which produces the bound stated in the proposition.

Method $[\Cent]$ concerns $O(n^2)$ instances of standard linear algebra
arithmetic on spaces of bounded dimension, and so its running time is dominated
again by time spent on the $[\mathrm A_2]$ method.

Method $[\Der]$ involves the computation of parts
of the Lie algebra of derivations. Computing the full Lie algebra of
derivations in instances like $D_n\SC(2)$ would take running time $\tO(n^{12}
\log(q))$. However, we only carry out this procedure for Lie algebras of
bounded dimension (the bound being 28, which occurs for type ${\mathrm D}_4$)
or compute the part of $\Der(L)$ that leaves invariant $H$ and the
corresponding decomposition into root spaces (which reduces the running time
to $\tO(n^8 \log(q))$).  Therefore, the stated bound suffices.

Finally, according to Table \ref{tab_multdimeigenspc}, Method $[\mathrm
B_2\SC]$ with unbounded $n$ only occurs in the cases treated in Section
\ref{sec_multdimsol_CnSC}, where the time analysis is already given.
\end{proof}

\section{Root Identification}\label{sec:RootIdent}
In this section we clarify Step 3 of the {\sc ChevalleyBasis} algorithm
\ref{alg_chevbas}. The routine {\sc IdentifyRoots} takes as input a Chevalley
Lie algebra $L$, a split Cartan subalgebra $H$ of $L$, the root datum $R$ and
the set of roots
$\overline\Phi = \Phi(L,H)$, 
and the Chevalley frame $\mathcal X$ found in the previous step (Section
\ref{sec_finding_frames}).
It returns a bijection $\iota: \Phi\to{\mathcal X}$
so that, up to scaling, $(X_\alpha)_{\alpha \in \Phi}$ will be the root element
part of a Chevalley basis.

An important tool to make this identification are the Cartan integers $\langle
\alpha, \beta\ch\rangle$.  Cartan integers may be computed using root chains;
see, for instance, \cite{Car72}.

\begin{lemma}\label{lem_rootchain}
Let $\alpha,\beta \in \Phi$. Suppose $p$ and $q$ are the largest non-negative
integers such that $\alpha-p\beta \in \Phi$ and $\alpha+q\beta \in \Phi$. Then
$\langle \alpha, \beta\ch \rangle = p-q$.
\end{lemma}

We may use this lemma by computing such a chain in the set of roots
$\overline{\Phi} $ corresponding to the Chevalley frame ${\mathcal X} = \{\F X_\alpha \mid
\alpha \in \Phi \}$.  However, as these roots are computed from the Lie
algebra $L$ over $\F$ itself, they live in the $n$-dimensional vector space
$\F^n$ rather than over $\mathbb Z^n$.

A straightforward verification of cases for Chevalley Lie algebras arising
from root systems of rank $2$ shows that the chain can simply be computed in
terms of the roots over $\F^n$, except if the characteristic is $2$ or $3$.
So in the latter two cases, we a different method for computing
$\langle \alpha,\beta\ch\rangle$ is needed.

\begin{lemma}\label{lem_rootchain_lieproduct}
Suppose that $L=L_\F(R)$ is a Chevalley Lie algebra with respect to a root
datum $R = (X,\Phi,Y,\Phi\ch)$ over the field $\F$ of characteristic $2$ or
$3$.  Let $H$ be the standard split Cartan subalgebra of $L$.  Suppose
furthermore that $X_\alpha, X_{-\alpha}, X_\beta, X_{-\beta}$ are four vectors
spanning root spaces corresponding to $\alpha, -\alpha, \beta, -\beta \in
\Phi$, respectively, and $\alpha \neq \pm \beta$.
\newcommand{\xa}{X_\alpha} \newcommand{\xma}{X_{-\alpha}}
\newcommand{\xb}{X_\beta} \newcommand{\xmb}{X_{-\beta}}

If $\Phi$ is simply laced, then $\langle \alpha, \beta\ch \rangle = P-Q$, where 
\[
P = \left\{ \begin{array}{ll} 
  0 & \mbox{ if } [\xmb,\xa] = 0 \\
  1 & \mbox{ if } [\xmb,\xa] \neq 0 
\end{array} \right. , \quad 
Q = \left\{ \begin{array}{ll} 
  0 & \mbox{ if } [\xb,\xa] = 0 \\
  1 & \mbox{ if } [\xb,\xa] \neq 0 
\end{array} \right. .
\]

If $\Phi$ is doubly laced and $\chr(\F) \neq 2$, then $\langle \alpha, \beta\ch \rangle = P-Q$, where
\[
\begin{array}{rcl}
P & = & \left\{ \begin{array}{ll} 
  0 & \mbox{ if } [\xmb,\xa] = 0 \\
  1 & \mbox{ if } [\xmb,\xa] \neq 0, [\xmb,[\xmb,\xa]] = 0 \\
  2 & \mbox{ if } [\xmb,[\xmb,\xa]] \neq 0 \\
\end{array} \right. \cr
Q & = & \left\{ \begin{array}{ll} 
  0 & \mbox{ if } [\xb,\xa] = 0 \\
  1 & \mbox{ if } [\xb,\xa] \neq 0, [\xb,[\xb,\xa]] = 0 \\
  2 & \mbox{ if } [\xb,[\xb,\xa]] \neq 0 \\
\end{array} \right. 
\end{array}
\]
\end{lemma}

\begin{proof}
\newcommand{\pgg}{p_{\gamma,\gamma'}} \newcommand{\qgg}{q_{\gamma,\gamma'}}
\newcommand{\Ngg}{N_{\gamma,\gamma'}} \newcommand{\Nab}{N_{\alpha,\beta}}
\newcommand{\Nba}{N_{\beta,\alpha}} \newcommand{\pab}{p_{\alpha,\beta}}
\newcommand{\qab}{q_{\alpha,\beta}} For any $\gamma, \gamma' \in \Phi$, let
$\pgg$ and $\qgg$ be the biggest non-negative integers such that $\gamma -
\pgg \gamma' \in \Phi$ and $\gamma + \qgg \gamma' \in \Phi$.  Recall from
(CB4) and \cite{Car72} that, if $\gamma+\gamma' \in \Phi$, then $[X_\gamma, X_{\gamma'}] = \Ngg
X_{\gamma+\gamma'}$, where $\Ngg=\pm (\pgg+1)$.

If $\Phi$ is simply laced, the subsystem of $\Phi$ generated by $\pm \alpha,
\pm\beta$ is of type $\mathrm A_1 \mathrm A_1$ or of type $\mathrm A_2$.  
Then $\alpha + \beta \in \Phi$ implies $\alpha - \beta
\not\in \Phi$, so $\Nab = \pm 1$. Similarly, $\Nba = \pm 1$.  This means
that, regardless of the characteristic, we can reconstruct $\pab$ and $\qab$
by the procedure described in the lemma, and thus compute $\langle \alpha,
\beta\ch \rangle = \pab - \qab$ by Lemma \ref{lem_rootchain}.

If $\Phi$ is doubly laced and $\chr(\F) \neq 2$, the subsystem of $\Phi$
generated by $\pm \alpha, \pm \beta$ is of type $\mathrm A_1 \mathrm A_1$,
$\mathrm A_2$, or $\mathrm B_2$. (Note that $\mathrm G_2$ never occurs inside
a bigger Lie algebra.)  In the first two cases the previous argument applies,
so assume $\pm \alpha, \pm \beta$ generate a subsystem of $\Phi$ of type
$\mathrm B_2$.  Similarly to the previous case, if $\alpha + \beta \in \Phi$
then $\alpha -2\beta \not\in \Phi$, so that $\Nab, \Nba \in \{ \pm 1, \pm
2\}$.  In particular, since $\chr(\F) \neq 2$, we find that both $\Nab$ and
$\Nba$ are nonzero, so that we can reconstruct $\pab$ and $\qab$ by the
procedure described in the theorem, and thus compute $\langle \alpha, \beta\ch
\rangle = \pab - \qab$ by Lemma \ref{lem_rootchain}.
\end{proof}

\begin{lemma}\label{lem_can_compute_cartint1}
Suppose that $L$ is a Chevalley Lie algebra over $\F$ with respect
to a root datum $R = (X,\Phi,Y,\Phi\ch)$, $H$ is a split Cartan subalgebra of
$L$, and $X_\alpha$ and $X_{\beta}$ are two root elements whose roots
with respect to $H$ are $\overline{\alpha}$ and $\overline{\beta}$
for certain $\overline\alpha, \overline\beta\in\overline\Phi$.
Suppose, furthermore, that one of the following statements holds.
\begin{enumerate}
\item\label{itm_lcc_char} $\chr(\F) \not\in \{2,3\}$,
\item\label{itm_lcc_simplaced} $\Phi$ is simply laced,
\item\label{itm_lcc_doubnot2} $\Phi$ is doubly laced and $\chr(\F) \neq 2$.
\end{enumerate}
Then $\langle \alpha, \beta\ch \rangle$ can be computed
from the available data
in $\tO(n^{10}\log(q))$ elementary operations.
\end{lemma}

\begin{proof}
Observe first of all that the case where $\alpha=\beta$ is easily caught, for example by computing $\dim(\langle \F X_{\alpha}, \F X_{\beta} \rangle_\F)$. Obviously then $\langle \alpha,\beta\ch\rangle=2$.

Moreover, we can distinguish the case where $\alpha = -\beta$ as follows.
If $\chr(\F) \neq 2$ we may simply test whether $\overline{\alpha} = -\overline{\beta}$. 
If on the other hand $\chr(\F) = 2$, we find the sets $\{ \{ \gamma, -\gamma\}
\mid \gamma \in \Phi^+\}$ as an auxiliary result of the algorithm {\sc
  FindFrames} described in introduction of Section \ref{sec_multdimsol_A2}.
If $\alpha = -\beta$, then of course $\langle \alpha, \beta\ch \rangle = -2$.

So assume $\alpha \neq \pm \beta$. Now if (\ref{itm_lcc_char}) holds we
compute $\langle \alpha,\beta\ch \rangle$ from the roots $\overline{\alpha}$
and $\overline{\beta}$ using Lemma \ref{lem_rootchain}, as mentioned earlier.

Suppose, therefore, (\ref{itm_lcc_simplaced}) or (\ref{itm_lcc_doubnot2})
holds.  We can find $\F X_{-\alpha}$ and $\F X_{-\beta}$ either simply by
considering $\{ \overline{\gamma} \mid \gamma \in \Phi \}$ (if $\chr(\F) \neq
2$) or as an auxiliary result of {\sc FindFrames} (if $\chr(\F) = 2$).  This
leaves us in a position where we may apply Lemma
\ref{lem_rootchain_lieproduct}, and thus find $\langle \alpha, \beta\ch
\rangle$.

Finally, the time needed does not exceed the time needed for standard linear
algebra arithmetic for each pair of roots, that is, $\tO(n^4\cdot
n^6\log(q))$.
\end{proof}

The last lemma enables us to compute Cartan integers in many cases.  For the
cases not covered by Lemma \ref{lem_can_compute_cartint1} we proceed as
follows to construct a direct identification $\iota$.

\begin{itemize}
\item $\mathrm B_n(2)$: The short root spaces generate an ideal, $I$ say, of
$L$ found by the Meat-axe, and the root eigenspaces of $H$ that do not lie in
$I$ belong to long roots.  These root spaces generate a subalgebra of type
$\mathrm D_n$.  This Lie algebra is simply laced, so the root identification
problem can be solved there.  This identifies the long root spaces. Now, for
$i = 1, \ldots, n$, let the short root $\gamma_i$ be $\alpha_i + \alpha_{i+1}
+ \cdots + \alpha_n$ and let $\alpha_0 = \alpha_1 + 2\alpha_2 + 2\alpha_3 +
\cdots + 2\alpha_n$ be the (long) highest root.  Observe then that $[
X_{\alpha_0}, X_{-\gamma_1} ] = X_{\gamma_2}$ and $[ X_{\alpha_0},
X_{-\gamma_2} ] = X_{\gamma_1}$, and $X_{-\gamma_1}$ and $X_{-\gamma_2}$ are
the only short root elements that do not commute with $X_{\alpha_0}$.  This
fact, together with the set of pairs $\{ \{ \gamma, -\gamma \} \mid \gamma \in
\Phi^+\}$ obtained in {\sc FindFrames}, allows us to find $X_{\pm \gamma_1}$
and $X_{\pm \gamma_2}$.  Note that we have to execute this procedure at most
twice, since there are only elements of $\mathcal X$ that could be identified
with $X_{-\gamma_1}$, and the other short root elements are fixed once
$X_{-\gamma_1}$ is fixed.

The other short root elements may now simply be found by using relations such as
$[X_{\gamma_i}, X_{-\alpha_i}] = X_{\gamma_{i+1}}$. 
\item $\mathrm C_n(2)$: The short root spaces generate an ideal of $L$ of type
$\mathrm D_n$, so we execute a similar procedure as in the previous case.
\item $\mathrm F_4(2)$: The short roots generate generate an ideal of $L$ of
dimension $26$ which together with the Cartan subalgebra $H$ gives a
$28$-dimensional subalgebra of type $\mathrm D_4$, allowing the same procedure
as before.
\item $\mathrm G_2(3)$: Similarly to the previous cases, we use the fact that
the short roots generate an ideal of $L$ of type $\mathrm A_2$, which is again simply
laced.
\item $\mathrm G_2(2)$: As described in Section \ref{sec_multdimsol_G22}, the
manner in which the root spaces in $L^{\mathrm A}$ correspond to those in $L$
is completely determined. Therefore, we may use the roots identified in
$L^{\mathrm A}$, which is simply laced, to identify the roots in
$L$.
\end{itemize}

These methods lead to the following conclusion.
\begin{proposition}\label{prop_identify_roots}
Given $L$ over $\F$, $H$, $R = (X, \Phi, Y, \Phi\ch)$, the set $\overline\Phi$
of roots of $H$ on $L$, and a Chevalley frame $\mathcal X $, the routine {\sc
IdentifyRoots} finds a bijection $\iota:
\Phi \rightarrow \mathcal X$ such that for all $\alpha,\beta \in \Phi$,
$\alpha \neq \pm \beta$,
\[
[\iota( \alpha ), \iota(\beta)]= \left\{ \begin{array}{ll} 
\iota(\alpha+\beta) &  \mbox{ if } \alpha+\beta \in \Phi \mbox{ and } N_{\alpha,\beta} \not\equiv 0\; (\mathrm{mod}\ p), \\
\{ 0\} & \mbox{otherwise.}
\end{array} \right.
\]
For $\F=\F_q$,
the routine needs $\tO(n^{10}\log(q))$ elementary operations.
\end{proposition}

\begin{proof}
Lemma \ref{lem_can_compute_cartint1} shows that in many cases we can compute
Cartan integers. To this end, we need to compute $\langle \alpha, \beta\ch
\rangle$ for all $O(n^4)$ pairs of roots, and every computation of this type
involves at most $6$ multiplications in $L$, requiring a total of
$\tO(n^{4+6} \log(q))$ elementary
operations.  Once these numbers are computed, it takes
$O(n^4)$ steps to select a set of simple roots and subsequently
to complete the bijection between $\Phi$ and
$\mathcal X$.  These last two steps use techniques similar to those
described by De Graaf \cite[Section 5.11]{deGraaf00}: the creation of a set
of simple roots $\Pi$ starts with taking an arbitrary root to be the first
member of $\Pi$.  We then iteratively pick a suitable additional simple root
$\beta$ having Cartan integer $\langle \alpha,\beta\ch\rangle \le0$ with the
members $\alpha$ of $\Pi$.  This proves that we can
make the required bijection in $\tO(n^{10} \log(q))$ time for the cases covered
by Lemma \ref{lem_can_compute_cartint1}.

For the remainder of the proof, we can restrict ourselves to the cases not
covered by Lemma \ref{lem_can_compute_cartint1}.  Here the procedure described
provides $\iota$ directly, so we only need prove the last assertion of the
proposition.  As $\mathrm G_2(2)$ is directly reduced to a case already
treated, it needs no further consideration. In each of the remaining cases, we
need to compute a subalgebra or an ideal of $L$. Although this is hard in
general, the fact that we have already found the Chevalley frame $\mathcal X$
and the fact that the subalgebra or ideal is a sum of elements from $\mathcal
X$ imply that the computations take
$\tO(n^{\maxexp}\log(q))$ elementary operations. 
A bijection $\iota'$ from
the relevant subsystem of $\Phi$ to the subset of $\mathcal X$ of
root spaces lying in the
ideal may then be identified in time $\tO(n^{10} \log(q))$.
Finally, extending $\iota'$
to the entirety of $\Phi$ is a straightforward task, requiring only 
standard linear algebra arithmetic in $L$.

This shows that we can make the required bijection in time stated for all
cases.
\end{proof}

\section{Conclusion}\label{sec_conclusion}

\begin{table}
	\[ \begin{array}{l|rrrr} R & \mathbb{Q} & 17 & 3^3 & 2^6 \cr
	\hline
	\hline
	\mathrm A_{1} & 0 & 0 & 0 & 0 \cr
	\mathrm A_{2} & 0 & 0 & 0 & 0 \cr
	\mathrm A_{3} & 0 & 0 & 0 & 0.7 \cr
	\mathrm A_{4} & 0.1 & 0 & 0.1 & 0.1 \cr
	\mathrm A_{5} & 0.1 & 0.1 & 0.1 & 0.2 \cr
	\mathrm A_{6} & 0.3 & 0.2 & 0.4 & 0.6 \cr
	\mathrm A_{7} & 0.6 & 0.5 & 0.9 & 1.5 \cr
	\mathrm A_{8} & 1.4 & 1 & 2 & 3.6 \cr
	\mathrm A_{9} & 2.8 & 2 & 4.2 & 7.9 \cr
	\mathrm B_{1} & 0 & 0 & 0 & 0 \cr
	\mathrm B_{2} & 0 & 0 & 0 & 0 \cr
	\mathrm B_{3} & 0 & 0 & 0.1 & 0.4 \cr
	\mathrm B_{4} & 0.1 & 0.1 & 0.2 & 1.9 \cr
	\mathrm B_{5} & 0.3 & 0.2 & 0.9 & 4.8 \cr
	\end{array} \;\;
	\begin{array}{l|rrrr} R & \mathbb{Q} & 17 & 3^3 & 2^6 \cr
	\hline
	\hline
	\mathrm B_{6} & 0.9 & 0.6 & 3.2 & 20 \cr
	\mathrm B_{7} & 2.2 & 1.6 & 10 & 54 \cr
	\mathrm B_{8} & 5.3 & 3.9 & 27 & 172 \cr
	\mathrm B_{9} & 12 & 8.8 & 68 & 493 \cr
	\mathrm C_{1} & 0 & 0 & 0 & 0 \cr
	\mathrm C_{2} & 0 & 0 & 0 & 0 \cr
	\mathrm C_{3} & 0 & 0 & 0.1 & 0.1 \cr
	\mathrm C_{4} & 0.1 & 0.1 & 0.2 & 1.1 \cr
	\mathrm C_{5} & 0.3 & 0.2 & 0.9 & 10 \cr
	\mathrm C_{6} & 0.9 & 0.6 & 3.2 & 40 \cr
	\mathrm C_{7} & 2.2 & 1.6 & 10 & 177 \cr
	\mathrm C_{8} & 5.2 & 3.9 & 27 & 693 \cr
	\mathrm C_{9} & 12 & 8.8 & 69 & 2212 \cr
	\cr
	\end{array} \;\;
	\begin{array}{l|rrrr} R & \mathbb{Q} & 17 & 3^3 & 2^6 \cr
	\hline
	\hline
	\mathrm D_{1} & 0 & 0 & 0 & 0 \cr
	\mathrm D_{3} & 0 & 0 & 0 & 0.3 \cr
	\mathrm D_{4} & 0.1 & 0.1 & 0.1 & 3.2 \cr
	\mathrm D_{5} & 0.2 & 0.1 & 0.3 & 22 \cr
	\mathrm D_{6} & 0.6 & 0.4 & 0.9 & 121 \cr
	\mathrm D_{7} & 1.6 & 1.1 & 2.8 & 545 \cr
	\mathrm D_{8} & 3.8 & 2.8 & 7.7 & 1994 \cr
	\mathrm D_{9} & 8.6 & 6.4 & 19 & 6396 \cr
	\mathrm E_{6} & 0.9 & 0.6 & 1.6 & 3.3 \cr
	\mathrm E_{7} & 4.1 & 3 & 11 & 27 \cr
	\mathrm E_{8} & 28 & 21 & 112 & 398 \cr
	\mathrm F_{4} & 0.2 & 0.2 & 0.7 & 3.3 \cr
	\mathrm G_{2} & 0 & 0 & 0.1 & 0.5 \cr
	\cr
	\end{array} \]
\caption{Algorithm \ref{alg_chevbas} timings}\label{tab_timings}
\end{table}

As discussed in Section \ref{sec_intro_alg_chevbas} the more difficult steps
of Algorithm \ref{alg_chevbas} are {\sc FindFrame} and {\sc IdentifyRoots}. In
Sections \ref{sec_finding_frames} (Proposition \ref{prop_FindFrame}) and
\ref{sec:RootIdent} (Proposition \ref{prop_identify_roots}) we established
that these steps can be dealt with in time $\tO(n^{\maxexp}\log(q)^4)$. This
proves Theorem \ref{th:main} for a given root datum.

We emphasize that this estimate is only asymptotic. Additionally, in Table
\ref{tab_timings} we present timings of Algorithm \ref{alg_chevbas} for
various root data and for four different fields: $\mathbb{Q}$, $\GF(17)$,
$\GF(3^3)$, and $\GF(2^6)$.  The times given are in seconds, for the most
time-consuming root datum with the specified Lie type.  Input for the
algorithm were a Chevalley Lie algebra and its splitting Cartan subalgebra, to
which a random basis transformation was applied, and the root datum.  The
timings were produced using \Magma\ 2.15-5 on an Intel Core 2 Quad CPU running
at 2.4 GHz with 8GB of memory available, although only one core and at most
2.7GB of memory were used.

\bigskip

As hinted at earlier, Algorithm \ref{alg_chevbas} can easily be used to produce an
algorithm that takes only $L$ and $H$ and produces the root datum $R$ and a
Chevalley basis.  To see this, note first that, because $H$ is given and the
underlying algebraic group is assumed to be simple, we may use $\dim(H) =
\rnk(R)$, the dimension of $L$, and the classification of simple Lie algebras
to narrow down the root system to one or two possibilities (or three, but only
if $\dim(L) = 78$ and $\dim(H) = 6$). Therefore, the number of possible root
systems it at most $3$.

Second, given a root system, the number of possible root data is small as
well.  If the root system is not of type $\mathrm A$ or $\mathrm D$, the
number of possible isogeny types is at most $2$.  If the root system is of
type $\mathrm D$ the number of possible isogeny types is at most $5$, as
explained in Section \ref{sec_multdim}.  So suppose $\Phi$ is of type $\An$,
and fix $p = \chr(\F)$. Note that the fundamental group is $\mathbb
Z/(n+1)\mathbb Z$.  Since two root data for $\An$ lead to isomorphic Lie
algebras if both have the same exponent of $p$ in $[X:\Z \Phi]$,
we need consider at most $\log_p(n+1) + 1=O(\log(n))$ different
isogeny types.  Thus, in order to identify the correct root datum,
we run Algorithm \ref{alg_chevbas} a sufficiently small
number of times for the polynomiality bound given in the theorem
to remain intact.
This finishes the proof of Theorem \ref{th:main}.

\bigskip
A problem hinted at, but not solved satisfactorily, is finding a split
Cartan subalgebra of a Chevalley Lie algebra $L$ if $p = 2$.
Nevertheless, verifying that a given subalgebra is indeed a split Cartan
subalgebra is easy.  So our results are still useful, since one is often able
to obtain such a subalgebra by other means, for example as part of the
original problem. 
Moreover, experimental implementations of randomized algorithms for finding
split Cartan subalgebras look promising.
We intend to publish about these algorithms in forthcoming work.

A primary goal in writing the Chevalley basis algorithm is to use it for
conjugacy questions in simple algebraic groups $G$ or finite groups $G(\F_q)$
of rational points over $\F_q$.  One of the complications in this application
is the fact that the group $\Aut(L)$ may be much larger than $G(\F_q)$.  For
this purpose, a method is needed to write an arbitrary automorphism of $L$ as
a product of an element from $G(\F_q)$ and a particular coset representative
of $G(\F_q)$ in $\Aut(L)$. Such a method is in \cite{CMT04} and is also used in
\cite{CM06}.

\bigskip
Once Algorithm \ref{alg_chevbas} completes succesfully we have a certificate
for a Lie algebra to be of type $R$: when presented with a candidate Chevalley
basis $X^0,H^0$, we only need to carry out the straightforward and quick task
of verifying that $X^0,H^0$ is indeed a Chevalley basis for $L$ with respect
to $H$ and $R$.  In this way, our work also contributes to a recognition
procedure for modular simple Lie algebras.  Obviously, Algorithm
\ref{alg_chevbas} can be used for establishing an isomorphism between two
Chevalley Lie algebras over the same field and of the same root datum.

\end{document}